\documentclass[aoap,preprint]{imsart}
\usepackage{amsmath,amsfonts,amssymb,amsthm,mathrsfs}
 
\usepackage{hyperref}

\begin{document}

\begin{frontmatter}

\title{Normal convergence of non-localised geometric functionals and shot noise excursions }
\runtitle{Non-localised geometric functionals}

\begin{aug}
  \author{\fnms{Rapha\"el  }  \snm{Lachi\`eze-Rey}  \ead[label=e1 ]{raphael.lachieze-rey@parisdescartes.fr}},

  \runauthor{R. Lachi\`eze-Rey}

  \affiliation{Universit\'e Paris Descartes, Sorbonne Paris Cit\'e}

  \address[a]{Universit\'e Paris Descartes, MAP5, 45 Rue des Saints-P\`eres, 75006 Paris, \printead{e1}}

\end{aug}
\begin{abstract}
This article presents a complete second order theory for a large class of geometric functionals on homogeneous Poisson input.  In particular, the results don't require the existence of a radius of stabilisation. Hence they can be applied to geometric functionals of spatial shot-noise fields excursions such as volume, perimeter, or Euler characteristic (the method still applies to stabilising functionals). More generally, it must be checked that a local contribution to the functional is not strongly affected  under a perturbation of the input far away. In this case the exact asymptotic variance is given, as well as the likely optimal speed of convergence in the central limit theorem.
This goes through a general mixing-type condition that adapts nicely to both proving asymptotic normality and   that variance is of volume order.

\end{abstract}


\end{frontmatter}

%
%
%

\numberwithin{equation}{section}
\theoremstyle{definition}

\newtheorem{theorem}{Theorem}[section]

\newtheorem{claim}{Claim}[section]

\newtheorem{proposition}{Proposition}[section]

\newtheorem{remark}{Remarks}[section]
\newtheorem{corollary}{Corollary}[section]

\newtheorem{lemma}{Lemma}[section] 
\newtheorem{example}{Example}[section] 
\newtheorem{assumption}{Assumption}[section]

\newcommand{\EC}{Euler characteristic }

\newcommand{\partialZ}{\partial_{\mathbb{Z} ^{d}}}
\newcommand{\Var}{\text{\rm{Var}}}
\newcommand{\cov}{\text{\rm{Cov}}}
\newcommand{\diam}{\text{\rm{diam}}}
\newcommand{\K}{\mathscr{K}}

\renewcommand{\div}{\text{\rm{div}}}

\newcommand{\magenta}{\color{magenta}}
\newcommand{\equlaw}{\stackrel{(d)}{=}}
\newcommand{\bft}{{\bf t}}
\newcommand{\z}{{\bf z}}
\newcommand{\y}{{\bf y}}
\newcommand{\subsetf}{{\subset^{*}} }


 \newcommand{\signe}{\text{\rm{sign}}}

 \renewcommand{\d}{{\rm d}}
 \newcommand{\W}{\mathscr{W}}
\renewcommand{\P}{\mathbf{P}} 
\newcommand{\E}{{\mathbf{E}}}

\newcommand{\var}{{\rm Var}}  
\newcommand{\leb}{\ell^d} 
 
\renewcommand{\S}{\mathcal{S}}
\newcommand{\Per}{{\rm Per}}

\newcommand{\Lip}{{\rm Lip}}

\newcommand{\Vol}{{\rm Vol}}

 \newcommand{\R}{\mathcal{R}}
\newcommand{\corners}{{\rm corners}}
 
\renewcommand{\v}{\mathbf{v}}
\renewcommand{\u}{\mathbf{u}} 
 
\newcommand{\A}{\mathcal{A}}
\renewcommand{\root}[1]{\underline{#1}}
 
\newcommand{\N}{\mathscr{N}}

\newcommand{\p}{\mathsf{P}}

\newcommand{\supp}{\text{\rm{supp}}}
\newcommand{\n}{\mathbf{n}}

\newcommand{\C} {\mathcal{C} }

 
%
%
%
%
%
%
%
%



\newcommand{\M}{{\bf M}}
\section{Introduction}

\newcommand{\x}{{\bf x}} 
Let $(\Omega ,\mathscr{A},\P)$ be a probability space.
Denote by $\ell^d$ the Lebesgue measure on $\mathbb{R}^{d}$.
Let $\eta $ be a homogeneous Poisson process on $\mathbb{R}^{d}$,  and $\{F_{W}(\eta );W\subset \mathbb{Z}^{d}\}$ a family of geometric functionals. We give  general conditions  under which $ F_{W}(\eta ) $ has a variance asymptotically proportional to $\sigma _{0}^{2} | W | $ for some   $\sigma _{0}>0$, and $ \Var(F_{W}(\eta ))^{-1/2}(F_{W}-\E F_{W}(\eta ))$ converges to a Gaussian variable, with a Kolmogorov distance   decaying  in $ | W | ^{-1/2}$, as $ | W |  $ goes to $\infty $.

{\bf Marked processes} The model is even richer if one marks the input points  by random independent variables, called \emph{marks}, drawn from an external  probability space $(\M,\mathscr{M}, \mu )$, the \emph{marks space}. It can be used for instance to let the shape and size of grains be random in the boolean model, or to have a random impulse function for a shot noise process. For $A\subset \mathbb{R}^{d}$, denote by  $\overline{A}=A\times \M$ the cylinder of marked points $\x=(x,m)$ with spatial coordinate $x\in A$. Endow $\overline{\mathbb{R}^{d}}$ with the product $\sigma $-algebra. The reader not familiar with such a setup can consider the case where $\M$ is a singleton, 	and all mark-related notation can be ignored (except in applications). By an abuse of notation, every spatial transformation applied to a couple $\x=(x,m)\in \overline{\mathbb{R}^{d}}$ is in fact applied to the spatial element, i.e. $\x-y=(x-y,m)$ for $y\in \mathbb{R}^{d},$ or for $A \subseteq \mathbb{R}^{d}\times \M,C\subset \mathbb{R}^{d}$, $A\cap C=\{(x,m)\in A:x\in   C\}$.
\newcommand{\F}{\mathcal F}
Denote for simplicity by $d\x=dx\mu (dm)$ the measure element on $(\overline{\mathbb{R}^{d}},\ell^d\times \mu )$. In all the paper, $\eta $ denotes a Poisson measure on $\overline{\mathbb{R}^{d}}$ with intensity measure $\ell^d\times \mu $.    We  assume that $\eta $ and all  random variables introduced in the paper live on the probability space $\Omega $, up to expanding it.
 
 {\bf Functionals} Let $\mathcal{A}$ be the class of locally finite sets of $\overline{\mathbb{R}^{d}}$ endowed with the topology  induced by the  mappings $\zeta \mapsto  | \zeta \cap A | $ for compact sets $A\subset \overline{\mathbb{R}^{d}}$, where $ | \cdot  | $ denotes the cardinality of a set. Functionals of interest are not properly defined on every $\zeta \in \mathcal{A}$, so we restrict them to some $\N_{0}\subset \mathcal{A} $ such that $\P(\eta \in \N_{0})=1$, and call $\N$ the class of configurations $\zeta \in \mathcal{A}$ such that $\zeta \subset \eta \cup \zeta '$ for some $\eta \in \N_{0}$ and  finite set $\zeta '$.    Let $\F$ be the class of real measurable functionals on $ \N$.  Let $\tilde Q_{a}=[-a/2,a/2)^{d},Q_{a}=\tilde Q_{a}\cap \mathbb{Z} ^{d},a>0$. For $W\subset  \mathbb{Z} ^{d}$ finite,
we  consider a functional of the form 
\begin{align}
\label{FW-general-form-Fk}
F_{W}(\zeta )=\sum_{k\in W}F_{k}^{W} (\zeta ),\zeta \in \N ,\text{\rm{ with }}F_{k}^{W}(\zeta )=F_{0}(\zeta \cap \tilde W-k), k\in W,\end{align}where $F_{0}\in \F$ and $\tilde W=\cup _{k\in W}(k+\tilde Q_1)$. It might also happen that all points of $\eta $ have an influence but only contributions of the functional over $\tilde W$ are considered: introduce the infinite input version
\begin{align}
\label{FW-general-form}
F_{W}'(\zeta )=\sum_{k\in W}F_{k} (\zeta ),\zeta \in \N,\text{\rm{ with }}F_{k} (\zeta )&=F_{0}(\zeta  -k),k\in \mathbb{Z} ^{d}.
\end{align} 
A score function is a bi-measurable mapping $\xi :\M\times \N\to \mathbb{R}$ such that
 \begin{align}
\label{eq:score-representation}
F_{0}^{\xi }:\zeta \mapsto \sum_{\x=(x,m)\in \zeta\cap {\tilde Q_1} }\xi (m,\zeta-x  ), 
\end{align} is well defined on $\zeta \in \N$, which yields that $F_{W}(\zeta )$ is the sum of the scores of all points falling in $\tilde W$.  Write $\xi (\zeta )$ instead of $\xi (m;\zeta )$ if no marking is involved (i.e. $\M$ is a singleton).   It is explained later why some shot noise  excursions functionals also obey  representations \eqref{FW-general-form-Fk}-\eqref{FW-general-form}.  In this paper, we identify a functional  $F:\N\to \mathbb{R}$   with the random variable that gives its value over $\eta : F=F(\eta )$, even if $F$ will be applied to modified versions of $\eta $ as well. 

{\bf Non-degeneracy of the variance} 
Define for $\zeta \subset \mathbb{R}^{d},0\leqslant a<b$,
\begin{align*}
\zeta _{a}^{b}=\zeta \cap \tilde Q_{b }\cap \tilde Q_{a}^{c},\; \zeta _{a}=\zeta \cap \tilde Q_{a}^{c},\;\zeta ^{b}=\zeta \cap \tilde Q_{b}.
\end{align*}

A condition that seems necessary for the variance to be non-degenerate is that at least on a finite input and a bounded window, the functional is not trivial: for some $\delta >\rho >0,$ $\P ( | F_{Q_{\delta  }}(\eta ^{\rho })-F_{Q_{\delta  }}(\emptyset ) | >0)>0$. We actually need that this still holds if points are added far away from $\eta ^{\rho }$:
 \begin{assumption}
\label{ass:nondeg}
There is  $\gamma >\rho >0,c>0,p>0$ such that for $\delta >\gamma $ arbitrarily large 
\begin{align*}
\P \left(
\left|
F_{Q_{\delta   }}(\eta _{\gamma })-F_{Q_{\delta  }}(\eta ^{\rho }\cup \eta _{\gamma })
\right|\geqslant c
\right)\geqslant p.
\end{align*} \end{assumption}

{\bf Observation window}
In many works (e.g. \cite{PenYuk},\cite[Chapter 4]{KenMol}), the observation windows consist in a growing family of subsets $B_{n},n\geqslant 1$ of $\mathbb{R}^{d}$, that satisfy the Van'Hoff condition: for all $r>0$, 
\begin{align}
\label{eq:VH}
\ell^d(\partial B_{n}^{\oplus r})/\ell^d(B_{n})\to 0,
\end{align} as $n\to \infty $, where $B^{\oplus r}=\{x\in \mathbb{R}^{d}:d(x,B)\leqslant r\}$ for $B\subset \mathbb{R}^{d}$. We rather consider in this paper, like for instance in \cite{SepYuk}, a  family $\W$ of bounded  subsets of $\mathbb{Z} ^{d}$  satisfying the regularity condition  
\begin{align}
\label{eq:ball-like}
\limsup_{W\in \W}\frac{ | \partialZ   W  | }{ |   W  | }= 0,
\end{align}
where $\partial_{\mathbb{Z} ^{d}}W$ is the set of points of $W $ at distance $1$ from $W ^{c}$, and consider a point process over $\tilde W$. In the large window asymptotics, condition \eqref{eq:ball-like} imposes the same type of restrictions as  \eqref{eq:VH}, and using subsets of the integer lattice eases certain estimates and is not fundamentally different.  In the case where boundary effects occur (by observing $\eta \cap \tilde W$ instead of $\eta $),  stronger geometric conditions will be required.  To this end, let $B_{r},r>0,$ be a family of measurable subsets of $\mathbb{R}^{d}$ such that for some $0<a_{-}<a_{+}$, $B(0,a_{-}r)\subset B_{r}\subset B(0,a_{+}r)$, where $B(x,r)$ is  the Euclidean ball with center $x\in \mathbb{R}^{d}$ and radius $r>0$. Let also $B_{r}(x)=x+B_{r},x\in \mathbb{R}^{d}$. We set similarly as in \cite[Section 2]{PenYuk},  \begin{align*}
\mathcal{B}_{W }^{r}=&\{ \tilde W-k:\;k\in    W ,B_{r}^{c}\cap (\tilde W-k)\neq \emptyset   \},W\subset \mathbb{Z} ^{d}, \\
\mathcal{B}_{\W}^{r}=&\bigcup _{W\in \W}\mathcal{B}_{W}^{r}\cup \{\mathbb{R}^{d}\}.
\end{align*}

{\bf Background} The family of functionals described above is quite general and covers large classes of statistics used in many application fields, from data analysis to ecology, see \cite{KenMol} for theory, models and applications.  We study the variance, and Gaussian fluctuations, of such functionals,  under the assumption that a modification of $\eta $ far from $0$ modifies slightly $F_{0}(\eta )$ (or $\xi (0,\eta  )$). Most of the general results available require a {\it stabilization} or {\it localisation} radius :  it consists in a random variable $R>0$, with sufficiently fast decaying tail, such that any modification of $\eta $ outside $B(0,R)$ does not affect $F_{0}(\eta )$ (or $\xi (0,\eta  )$) at all.  By stationarity this behaviour is transferred to any $F_{k},k\in \mathbb{Z} ^{d}$. This property is sometimes called {quasi-locality} in statistical physics \cite{RVY}.  In the Euclidean framework, the results of the present paper do not require stabilisation, but can still be applied to geometric functionals, see Section \ref{sec:stab-NN}.

We give  general conditions under which functionals of the form \eqref{FW-general-form-Fk}-\eqref{FW-general-form} have a volume order variance  and undergo a central limit theorem, with  a Kolmogorov distance to the normal given by the inverse square root of the variance. We recall that the Kolmogorov distance between two real variables $U$ and $V$ is defined as 
\begin{align}
\label{eq:dK-def}
d_\K  (U,V)=\sup_{t\in \mathbb{R}}\left|
\P(U\leqslant t)-\P(V\leqslant t)
\right|.
\end{align} Specified to the case where functionals  are under the form \eqref{eq:score-representation} and the score function is stabilizing,  our conditions  demand that the tail of the stabilization radius $R$ decays polynomially fast, with power strictly smaller than   $-8d$, see Proposition 
\ref{thm:stab}.

{\bf Main result}  The main theoretical finding of this paper is condition \eqref{eq:ultimate-assumption-bis}, which is well suited for second order Poincar\'e inequalities in the Poisson space, i.e. bounds on the speed of convergence of a Poisson functional to the Gaussian law, and at the same time allows to prove non-degenerate asymptotic variance under  Assumption \ref{ass:nondeg}. The application to shot-noise processes in the following section illustrates the versatility of the method.
The results can be merged into the following synthetic  result, whose proof is at Section \ref{sec:proof-main}.
 For two sequences $\{a_{n};n\geqslant 1\},\{b_{n};n\geqslant 1\}$, write $a_{n}\sim b_{n}$ if $b_{n}\neq 0$ for $n$ sufficiently large and $a_{n}b_{n}^{-1}\to 1$ as $n\to \infty .$ Also, in all the paper, $\kappa $ denotes a constant that depends on $d,\alpha ,a_{+},a_{-}$,  whose value may change from line to line, and which explicit optimal value in the main result could be traced through the different parts of the proof. If it is well defined, for $F_{0}\in \F,$ let
 \begin{align}
\label{eq:sigma0}
\sigma _{0}^{2}:=\sum_{k\in \mathbb{Z} ^{d}}\cov(F_{0}(\eta ),F_{k}(\eta )).
\end{align}

 \begin{theorem}
\label{thm:ultimate}
  Let 
  $F_{0}\in \F$, $F_{W} $ be defined as in \eqref{FW-general-form-Fk}, $\W=\{W_{n};n\geqslant 1\}$  satisfying   \eqref{eq:ball-like}. Let $M_{1},M_{2}$ be independent random elements of $\M$ with law $\mu.$ Assume that for some $C_{0}>0,\alpha >2d $,   
 {   for all $r\geqslant 0, B \in \mathcal{B}_{\W}^{r},\ell^d-a.e.\; x_{1},x_{2}\in \mathbb{R}^{d},\zeta \subset  \{(x_{1},M_{1}),(x_{2},M_{2})\} $ 
 \begin{align}
\label{eq:ultimate-assumption-bis} \left(
\E 
   \left|F_{0}((\eta\cup \zeta ) \cap B_{r}\cap B )-F_{0}((\eta\cup \zeta) \cap B    )
\right| ^{4}
\right)^{1/4}
 \leqslant C_{0}(1+r)^{-\alpha }, 
\end{align}  }
and  Assumption \ref{ass:nondeg}   is satisfied.
Then $0<\sigma _{0}<\infty $, and as $n\to \infty ,$
 $$
{\var(F_{W_{n}})}\sim \sigma _{0}^{2} | W_{n} |
,\hspace{1cm}(\sigma _{0} ^{2}| W _{n}|)^{-1/2}(F_{ {W_{n}}}-\E F_{ { {W_{n}}}}) \xrightarrow[n\to \infty ]{\text{\rm{law}}}\; N$$ where $N$ is a standard Gaussian random variable. Furthermore, for $n$ sufficiently large,
\begin{align}
\label{eq:ultimate-dK}
d_\K  \left(
 {  \frac{ F _{W_{n}}-\E F _{W_{n}}}{\var(F _{W_{n}})^{1/2}}},N
\right)
 \leqslant \kappa  | W _{n}| ^{-1/2}\left(
\frac{C_{0}^{2}}{\sigma _{0}^{ 2}}+\frac{C_{0}^{3}}{\sigma _{0}^{ 3}}+ \frac{ C_{0}^{4}  }{\sigma _{0}^{ 4}} 
\right).
\end{align} 
\end{theorem}

Let us now give the version with infinite input, which is more simple to satisfy due to the absence of boundary effects, except for the power  of the decay:
 \begin{theorem}
\label{thm:ultimate-2}
  Let 
  $F_{0}\in \F$, $F_{W}' $ be defined as in \eqref{FW-general-form},  $\W=\{W_{n};n\geqslant 1\}$  satisfying   \eqref{eq:ball-like}. Let $M_{1},M_{2}$, be independent random elements of $\M$ with law $\mu.$ Assume that for some $C_{0}>0,\alpha >5d/2 $,  
 for all $r\geqslant 0,\ell^d-a.e. \;x_{1},x_{2}\in \mathbb{R}^{d},\zeta \subset  \{(x_{1},M_{1}),(x_{2},M_{2})\} $ ,
 \begin{align}
\label{eq:ultimate-assumption} \left(
\E 
   \left|F_{0}((\eta\cup \zeta ) \cap B_{r} )-F_{0}(\eta\cup \zeta    )
\right| ^{4}
\right)^{1/4}
 \leqslant C_{0}(1+r)^{-\alpha }, 
\end{align}  
and    Assumption \ref{ass:nondeg} is satisfied.
Then $0<\sigma _{0}<\infty $ (defined in \eqref{eq:sigma0}), and
 $$
{\var(F'_{ {W_{n}}})}\sim \sigma _{0}^{2} | W_{n} |,
\hspace{1cm}(\sigma _{0} ^{2}| W _{n}|)^{-1/2}(F'_{ {W_{n}}}-\E F'_{ {W_{n}}}) \xrightarrow[]{\text{\rm{law}}}\;N$$  as $n\to \infty .$
Furthermore,   for $n$ sufficiently large
\begin{align}
\label{eq:ultimate-dK-2}
d_\K  \left(
 {  \frac{ F' _{W_{n}}-\E F '_{W_{n}}}{\var(F' _{W_{n}})^{1/2}}},N
\right)
 \leqslant \kappa  | W _{n}| ^{-1/2} \left(
\frac{C_{0}^{2}}{\sigma _{0}^{ 2}}+\frac{C_{0}^{3}}{\sigma _{0}^{ 3}}+\frac{{ C_{0}^{ 4} } }{\sigma _{0}^{ 4}} 
\right).
\end{align} 
\end{theorem}
  
  \begin{remark}
\label{rk:main}
\begin{enumerate}
\item 
 The application to score functionals (see \eqref{eq:score-representation}) goes as follows:  let $M_{i},0\leqslant i\leqslant 6$ be iid marks with law $\mu $, and assume that $\xi :\M\times \N\to \mathbb{R}$ satisfies 
 for all $r \geqslant 0 ,B\in \mathcal{B}_{\W}^{r}, x_{0}\in \tilde Q_1  ,\zeta  \subset  \overline{\mathbb{R}^{d}}$ with at most $6$ elements,    \begin{align}
\label{eq:ultimate-assumption-score}
 \left(
\E 
  \left|\xi (M_{0},(\eta\cup \zeta   ) \cap B \cap  B_{r}-x_{0})-\xi (M_{0},(\eta\cup \zeta)\cap B-x_{0}  )
\right|  ^{4}
\right)^{1/4}
 \leqslant C_{0} (1+r )^{-\alpha },
\end{align} 
 then the functional $F_{0}=F_{0}^{\xi }$ defined in \eqref{eq:score-representation} satisfies  \eqref{eq:ultimate-assumption-bis}.  To see it, let $\x_{i}=(x_{i},M_{i})$ be the elements of $\zeta $. Fix $\zeta _{1}\subset \{(x_{1},M_{1}),(x_{2},M_{2})\}$,{ apply Lemma \ref{lm:mecke} (with $r =0$) to }
\begin{align*}
\psi ((x_{0},M_{0}),\zeta' )=\mathbf{1}_{\{x_{0}\in \tilde Q_1 \}}&\left| 
\xi (M_{0},  (\zeta' \cup \zeta _{1}) \cap B\cap B_{r}-x_{0})\right.\\
&\left.-\xi (M_{0}, ( \zeta'\cup \zeta _{1})\cap B-x _{0})
\right|,\zeta' \in \N,x_{0}\in \mathbb{R}^{d}.\end{align*} 
It yields 
\begin{align*}
\left(
\E\left|
F_{0}((\eta \cup \zeta_{1}) \cap B\cap B_{r})-F_{0}((\eta \cup \zeta_{1})\cap B )
\right|^{4}
\right)^{1/4}\leqslant &\left(
 \E\left|
\sum_{\x\in \eta \cap \tilde Q_1 }\psi (\x,\eta   )
\right|^{4}
\right)^{1/4}\\
\leqslant & \kappa C_{0} (1+r)^{-\alpha }
\end{align*}for some $C_{0} \geqslant 0$, hence \eqref{eq:ultimate-assumption-bis} is satisfied.
{  In this framework the asymptotic  variance can also be expressed as 
\begin{align*}
\sigma _{0}^{2}=\E \xi (M_{0};\eta )^{2}+\int_{ {\mathbb{R}^{d}}}(\E[\xi (M_{0},\eta \cup \{(x,M_{1})\}) \xi (M_{1},\eta \cup \{(0,M_{0})\}-x )] -[\E[\xi (M_{0};\eta )]]^{2})d x,
\end{align*}} see for instance (4.10) in \cite{KenMol}.
\item A variant of stabilisation, called {\it strong stabilisation}, occurs when    the add-one cost version of the functional is stabilising instead of the functional itself. Penrose and Yukich derived variance asymptotics and asymptotic normality \cite{PenYuk}  in such a context.  Let us indicate how the current approach could be adapted to strong stabilisation:  let $\eta '$ be an independent copy of $\eta $, and for $r>0$, $\eta _{r}=(\eta \cap B_{r})\cup (\eta '\cap B_{r}^{c})$. Assume that a functional has a strong stabilisation radius with the tail decaying as a sufficiently low power of $r$. In this case,  \eqref{eq:ultimate-assumption-bis} needs to hold with the left hand member replaced with
 $
\E \left(
 | F_{0}((\eta \cup \zeta )\cap B)-F_{0}((\eta _{r}\cup \zeta) \cap B) | ^{4}
\right).
 $ Then it should be possible to adapt  the proofs of Theorems \ref{thm-variance} and \ref{thm:BE} to be able to prove that the Berry-Esseen bounds and variance upper bounds   still hold, under this new hypothesis. 
\item
Regarding variance asymptotics, recent results  can be found in  the  literature, but the assumptions are of  different nature, either dealing with different qualitative long range behaviour (i.e. strong stabilization in \cite{PenYuk,LSY}), or  different non-degeneracy statements \cite{LPS}, whereas Assumption \ref{ass:nondeg} is 
a mixture of non-triviality and continuity of the functional on large inputs.  Penrose and Yukich \cite{PenYuk}  give a condition under which the asymptotic variance is strictly positive in Theorem 2.1. The condition is that the functional is {\it strongly stabilising}, and that the variable
\begin{align*}
\Delta (\infty ):=\lim_{\delta \to \infty }[F_{Q_{\delta }}(\eta\cup \{0\} )-F_{Q_{\delta }}(\eta  )]
\end{align*}
is non-trivial. It roughly means that  for $\delta $ sufficiently large, and $\rho $ sufficiently small, 
\begin{align*}
 \text{\rm{Var}}(|F_{Q_{\delta }}(\eta _{\rho }\cup \eta ^{\rho })-F_{Q_{\delta }}(\eta _{\rho })\;|\;|\eta ^{\rho }|=1\;)>0,
\end{align*}
and this is very close to Assumption \ref{ass:nondeg} in the particular case $\rho =\gamma $. This particular case seems more delicate to deal with that when $\gamma $ is much larger than $\rho $, because in the latter case the interaction between $\eta ^{\rho }$ and $\eta _{\gamma }$ hopefully becomes small.
\item Similar results where the input consists of $m_{n}$ iid variables uniformly distributed in $\tilde W_{n}$, with $m_{n}= | W_{n} | $, should be within reach by applying the results of \cite{LRP}, following a route similar to   \cite{LSY}.
\end{enumerate}
\end{remark}

{\bf Shot-noise excursions}  Let $\{g_{m};m\in \M\}$ be a set of measurable functions $\mathbb{R}^{d}\to \mathbb{R}$ not containing the function $g\equiv 0$ indexed by some probability space $(\M,\mathscr{M},\mu )$. Let  $\eta $  be a Poisson process with intensity measure $\ell^d\times \mu $ on $\overline{\mathbb{R}^{d}}$. Introduce the   shot noise processes with {\it impulse distribution} $\mu $ by, for $\zeta \in \N, $
\begin{align}
\label{def-SN}
f_{\zeta  }(y  )= &\sum_{\x=(x,m)\in \zeta  }g_{m}(y-x),y\in \mathbb{R}^{d}.
\end{align} 
Conditions under which $f_{\zeta  }$ is well defined on Poisson input are discussed in Section \ref{sec:SN}, along with a proper choice for $\N_{0}$. 
   Given some threshold $u\in \mathbb{R}$, we consider the  excursion set
$ 
\{f_{\zeta }\geqslant u\}=\{x\in \mathbb{R}^{d}:f_{\zeta }(x   )\geqslant u\} 
 $ and the functionals
 $
 \zeta \mapsto \ell^d(\{f_{\zeta }\geqslant u\}\cap \tilde W) $, $
 \zeta \mapsto \Per( \{f_{\zeta }\geqslant u\} ; \tilde W),
$ where for $A,B\subset \mathbb{R}^{d}; \Per(A;B)$ denotes the amount of perimeter of $A$ contained in $B$  in the variational sense, see Section \ref{sec:SN-per}.  The total curvature, related to the Euler characteristic is also studied in Section \ref{sec:sn-general} for a specific form of the kernels.

  A shot noise field is the result of random functions translated at random locations in the space. It has been introduced by Campbell to model thermionic noise \cite{campbell},  and has been used since then under different names  in many fields such as pharmacology, mathematical morphology \cite[Section 14.1]{Lan02},  image analysis \cite{GaborNoise},  or telecommunication networks \cite{BaccelliBiswas,BaccelliBook}.  Bierm\'e and Desolneux \cite{BieDes12,BieDesEuler,BieDes} have computed the mean values for some geometric properties of excursions. More generally, the activity about asymptotic properties of random fields excursions has recently increased, with the notable recent contribution of Estrade and L\'eon \cite{EstLeo}, who derived a central limit theorem for the Euler characteristic of excursions of stationary Euclidean Gaussian fields.
   Bulinski, Spodarev and Timmerman \cite{BST}  give general conditions for asymptotic normality of the excursion volume for quasi-associated random fields. Their results apply  to shot-noise fields, under   conditions of non-negativity and uniformly  bounded marginal density, which can be verified in some specific examples. We give  here the asymptotic variance and  central limit theorems for volume and perimeter of excursions under weak assumptions on the density,   as illustrated in Section \ref{sec:SN}. Still,  a certain control of the distribution is necessary, and we provide  in Lemma \ref{lm:sn-density} a uniform bound on $\sup_{v\in \mathbb{R},\delta >0}(\delta \ln(\delta ))^{-1}\P(f_{\eta }(0 )\in [v-\delta ,v+\delta ])$  when $f$ is of the form  
\begin{align}
\label{eq:sn-g}
f_{\zeta }(x )=\sum_{i\in I}g(\|x-x_{i}\|)
\end{align}where  $\zeta \in \N$, and $x_{i},i\in I,$ are the (random) spatial locations of its points, with $g$ a smooth  strictly non-increasing  function  $ (0,\infty )\to (0,\infty )$ with a derivative not decaying too fast to $0$.
Our results allow  to treat fields with singularities, such as those observed in astrophysics or telecommunications, see  \cite{BaccelliBiswas}.
 
  Let $\mathcal{M}_d $ be the space of measurable subsets of $\mathbb{R}^{d}$. The results of Section \ref{sec:SN} also apply  to processes that can be written under the form 
\begin{align}
\label{eq:sn-a}
f_{\zeta }(x )=\sum_{i\geqslant 1}L_{i}\mathbf{1}_{\{x-x_{i}\in A_{i}\}},x\in \mathbb{R}^{d},
\end{align}where the $(L_{i},A_{i}),i\geqslant 1$ are iid couples of $\mathbb{R} \times \mathcal{M}_d $, endowed with a proper $\sigma $-algebra and probability measure, see Section \ref{sec:sn-general}. Such models are called \emph{dilution functions} or \emph{random token models} in mathematical morphology, see for instance \cite[Section 14.1]{Lan02},  where they are used to simulate random functions with a prescribed covariance. 

To the best of our knowledge, the results about the perimeter or the Euler characteristic are the first of their kind for shot noise models, and the results about the volume improve existing results, see the beginning of Section \ref{sec:SN-volume} for more details.
%

\subsection{Stabilization and nearest neighbour statistics}
\label{sec:stab-NN}

Let us transpose our results in the case where the functional stabilises.

\begin{theorem}
\label{thm:stab}Let $\W=\{W_{n};n\geqslant 1\}$ be a class of subsets of $\mathbb{Z} ^{d}$.
Let $F_{W}$ be defined as in \eqref{FW-general-form-Fk} (resp. as in \eqref{FW-general-form-Fk}-\eqref{eq:score-representation} with $F_{0}=F_{0}^{\xi }$ for some score function $\xi $). Assume that
  for $x_{i}\in \mathbb{R}^{d}, M_{i} $ independent with law $\mu ,i\geqslant 1 ,   \zeta \subset \{(x_{i},M_{i});i=1,\dots ,6\} ,\eta '=\eta \cup \zeta $,  there is a random variable $R\geqslant 0$ such that   almost surely, for $r\geqslant R,B\in \mathcal{B}_{\W}^{r},$ 
\begin{align}
\label{eq:classic-stab}
F_{0}(\eta '  \cap  B_{r}\cap B)=&F_{0}(\eta' \cap B ).\\
\label{eq:classic-stab-score}
( \text{\rm{resp. }}\xi  (m,\eta '\cap B_{r}\cap B-x )=&\xi (m,\eta '\cap B-x ), (x ,m)\in \eta \cap\overline{ \tilde Q_1 }.)
\end{align}Then \eqref{eq:ultimate-assumption-bis}
is satisfied if for some $p,q>1$ with $1/p+1/q=1$, $\mathbb{P}(R>r)\leqslant C r^{-8dp-\varepsilon }$ for some $C,\varepsilon >0$, under the moment condition 
\begin{align}
\notag\sup_{r\geqslant 0,B\in \mathcal{B}_{\W}^{r}}\E\left|
F_{0}(\eta '\cap B\cap B_{r})
\right|^{4q}<\infty\\
\label{eq:moment-stab}
 (\text{\rm{resp. }}\sup_{r\geqslant 0,B\in \mathcal{B}_{\W}^{r},x_{0}\in \tilde Q_1 }\E\left|
\xi (M_{1},\eta '\cap B\cap B_{r}-x_{0})
\right|^{4q}<\infty ).
\end{align}

For the infinite input version, ``$\cap B$'' should be removed from \eqref{eq:classic-stab} (resp. \eqref{eq:classic-stab-score}),    the exponent $-8dp-\varepsilon $ should be replaced by $-10dp-\varepsilon $, and then \eqref{eq:ultimate-assumption} would hold.
\end{theorem}  

  \begin{proof}
 For $r\geqslant 0,B\in \mathcal{B}_{\W}^{r},$ if \eqref{eq:classic-stab} holds,
  \begin{align*}
\E\left|
F_{0}(\eta '\cap B)-F_{0}(\eta '\cap B\cap B_{r})
\right|^{4}=& \E\mathbf{1}_{\{R>r\}}\left|
F_{0}(\eta '\cap B)-F_{0}(\eta '\cap B\cap B_{r}))
\right|^{4}\\
\leqslant &\P(R>r)^{1/p}\left(\E\left(
 | F_{0}(\eta '\cap B\cap B_{r}) | 
+ | F_{0}(\eta '\cap B) | 
\right)^{4q}\right)^{1/q},
\end{align*}
hence \eqref{eq:ultimate-assumption-bis} is satisfied. If $F_{0}=F_{0}^{\xi }$, and \eqref{eq:classic-stab-score} holds, for $r\geqslant R$
\begin{align*}
F_{0}^{\xi }(\eta '\cap B_{r}\cap B)=&\sum_{(x,m)\in \eta \cap \tilde Q_1 }\xi (m,\eta '\cap B_{r}\cap B-x)\\
=&\sum_{(x,m)\in \eta \cap \tilde Q_1 }\xi (m,\eta ' \cap B-x)\\
=&F_{0}^{\xi }(\eta '\cap B),
\end{align*}and \eqref{eq:classic-stab} holds.
\end{proof}

\begin{remark} \begin{enumerate}
\item The variance non-degeneracy is 	a disjoint issue, Assumption \ref{ass:nondeg} has to be satisfied independently. Otherwise, if one is only interested in asymptotic normality, the above requirements can be weakened, see Theorem \ref{thm:BE}. 
\item The definition of a stabilisation radius often involves stability under the addition of an external set, here denoted by $\zeta $. A nice aspect of \eqref{eq:classic-stab}-\eqref{eq:classic-stab-score} with respect to classical results is that $\zeta $ does not depend on $\eta $, i.e. $\zeta $ does  not in general  achieve the worst case scenario given $\eta $. 
 On the other hand,   in the finite input version, one has to deal here with the intersection with $B\in \mathcal{B}_{\W}^{r}$. See Example \ref{ex:NN} for an application to nearest neighbour statistics.
  \item Asymptotic results for stabilizing functionals have been derived in numerous work, see the survey  \cite[Chapter 4]{KenMol} and references therein. In particular, Matthew Penrose first proved such results under polynomial decay for the stabilisation radius. 
 \end{enumerate}
\end{remark}

\begin{example}[Nearest neighbours statistics] 
\label{ex:NN} 
Let us develop the example of nearest neighbour statistics for illustrative pruposes.
Given $\zeta \in \N,x\in \mathbb{R}^{d} $, denote by  $N\hspace{-.5mm}N(x;\zeta )$ the \emph{nearest neighbour} of $x$, i.e. the closest point of $\zeta \setminus \{x\}$ from $x$, with ties broken by the lexicographic order.  Define recursively, for $k\geqslant 1$, $N\hspace{-.5mm}N_{k}(x;\zeta )=N\hspace{-.5mm}N(x;\zeta \setminus \cup _{i=0}^{k-1}N\hspace{-.5mm}N_{i}(x;\zeta ))$, with $x=N\hspace{-.5mm}N_{0}(x;\zeta )$, and $N\hspace{-.5mm}N_{\leqslant k}(x;\zeta )=\cup _{i=0}^{k}N\hspace{-.5mm}N_{i}(x;\zeta )$. Fix $k\geqslant 1$ and call {  neighbours} of $x$ within $\zeta $ the set $N_{k}(x;\zeta )$ consisting of all  points $y\in \zeta $ such that $x\in N\hspace{-.5mm}N_{\leqslant k}(y,\zeta \cup \{x\})$ or $y\in N\hspace{-.5mm}N_{\leqslant k}(x;\zeta )$. 

Let then $\varphi $ be a real functional defined on finite subsets of $\mathbb{R}^{d}$, and define the score function, for $\zeta\in \N$,
\begin{align*}
 \xi (\zeta )=\begin{cases}
\varphi (N_{k}(0;\zeta ))\\
 0$ if $ | \zeta  | < k. 
\end{cases}
\end{align*} Assume that for each $j\geqslant k,$ the induced mapping on $(\mathbb{R}^{d})^{j}$, $\tilde \varphi _{j}: (x_{1},\dots ,x_{j})\mapsto \varphi (\{x_{1},\dots ,x_{j}\})$, is measurable. The simplest example would be for $k=1$ the functional $\varphi (A)=\frac{1}{2}\sum_{y\in A }\|y\|$, so that
$ 
F_{W}(\zeta )=\sum_{x\in \zeta }\xi (\zeta -x)
 $ gives the total length of the undirected nearest-neighbour graph for $\zeta \subset \tilde W $. 
 Notice that no marking is involved in this setup.
  Such statistics are used in many applied fields,  in  nonparametric estimation procedures, or  more recently in estimation of high-dimensional data sets \cite{LevBic}. Many asymptotic results have been established since the central limit theorem of  Bickel and Breiman \cite{BicBre}, see for instance \cite{LSY,LPS,PenYuk}.
 
\begin{theorem}For $n\geqslant 1$, let  
\begin{align*}
G_{n}=\sum_{x\in \eta\cap \tilde Q_{n^{1/d}} }\varphi (N_{k}(x;\eta \cap \tilde Q_{n^{1/d}})).
\end{align*}Assume that there is $C,c>{0},u<d/4$ such that for all $x_{1},\dots ,x_{m  }\in \mathbb{R}^{d}$,
\begin{align}
\label{eq:bound-NN}
\varphi (\{x_{1},\dots ,x_{m   }\})\leqslant C\exp(c\max_{i}\|x_{i}\|^{u})
\end{align}and  that $\varphi $ is not degenerate: $\varphi (\{x_{1},\dots ,x_{k}\})\neq 0$ for   $(x_{1},\dots ,x_{k})$ in a non-negligible subset of $(\mathbb{R}^{d})^{k}$.
Then  $\var(G_{n})\sim  n \sigma _{0}^{2}$,  with $\sigma _{0}>0$ defined in Remark \ref{rk:main}, and $n^{-1/2}(G_{n}-\E G_{n})  $ converges in law to a centred Gaussian variable with variance $\sigma _{0}^{2}$, with bounds on the Kolmogorov distance proportional to $ n ^{-1/2}$.\end{theorem}


\begin{proof} 
Call hypercube a set of the form $x+[-a,a]^{d}$ for some $x\in \mathbb{R}^{d},a\geqslant 0.$
 For this proof we   choose  $B_{r}=[-r,r]^{d},r\geqslant 0$ (hence $a_{-}=1,a_{+}=\sqrt{d}$). Let  $a_{0}\in (0,1/4)$ and  $Q_{i}=x_{i}+[-a_{0},a_{0}]^{d},i=1,\dots ,q$ be hypercubes  contained in $B_{1}\setminus B_{1/2\sqrt{d}}$ such that the following holds: for all hypercube $B$ that touches $B_{1/2\sqrt{d}}$ and $B_{1}^{c}$ and $y\in B\cap  B(0,1)^{c}$, there is $i$ such that $Q_{i}\subset  (B\cap B(y,\|y\|)$). Let $Q_{i}'=x_{i}+[-a_{0}/2,a_{0}/2]^{d}$ and
\begin{align*}
R=\min\{r\geqslant 2\sqrt{d}(1+1/a_{0}) : | \eta \cap rQ_{i}' | \geqslant k\text{\rm{ for every }}i=1,\dots ,q \}.
\end{align*}The fact that $R':=\sqrt{d}(R+1)$
is a stabilization radius in the sense of \eqref{eq:classic-stab-score} is implied by the following claim:
\begin{claim}Let $r\geqslant R'$,
 $ B\in B_{\W}^{r},x\in B_{1} $.   All elements of $N_{k}(0,\eta '\cap B-x)$ are in   $B(0,\sqrt{d}R)$.   \end{claim}
\begin{proof} 
 Let $y\in \eta '\cap (B-x)$ be such that $0\in N\hspace{-.5mm}N_{\leqslant k}(y,(\eta ' \cap B-x)\cup \{0\})$. Assume that $y\notin B(0,R)$, hence $y \in (B-x) \cap B(0,R)^{c}$. Since $B\cap B_{r}^{c}\neq \emptyset $, $(B-x)\cap B^{c}_{r-\sqrt{d}}\neq \emptyset $, and $(B-x)\cap B_{R}^{c}\neq \emptyset $.  $0\in B$  yields $(B-x)\cap B_{t}\neq \emptyset $ for $t\geqslant 1$, hence for $t=R/2\sqrt{d}$. It follows that there is $i$ such that $B(y,\|y\|)\cap (B-x)$ contains $RQ_{i}$. Since $\eta $ has (at least) $k$ points in $RQ_{i}'$ and $RQ_{i}'-x\subset RQ_{i}$ (using $Ra_{0}/2\geqslant \sqrt{d}$), $\eta -x$ has $k$ points in $RQ_{i}$, hence $(\eta ' \cap B-x)\cap B(y,\|y\|)$ contains at least $k$ points, and they are all closer from $y$ than $0$, which contradicts $0\in N\hspace{-.5mm}N_{\leqslant k}(y,(\eta '\cap B-x)\cup \{0\})$. This proves $y\in B(0,R)$. 
 
 For every $i$, $RQ_{i}$ contains $k$ points of $\eta $ that are in $B_{R}$, hence in $B(0,R')$, hence $N\hspace{-.5mm}N_{\leqslant k}(0,\eta '\cap B-x)\subset B_{R'}$.
\end{proof}

The claim implies that $N_{k}(0,\eta '\cap B-x)=N_{k}(0,\eta '\cap B\cap B_{r}-x)$ for $r\geqslant R'.$
We have for $r\geqslant 0$, 
\begin{align*}
\P(R\geqslant r)\leqslant&\sum_{i=1}^{q} \P( | \eta \cap rQ_{i}' | \leqslant k-1)\leqslant \lambda r^{(k-1)d}e^{- \lambda' r^{d}}
\end{align*} (for some $\lambda,\lambda ' >0)$, and a similar bound holds for $R'$.
For the moment condition, note that for $r >0$, the neighbours of $0$ in $\eta \cap B_{r}\cap B-x$ are at most at distance $R' $, hence, in virtue of \eqref{eq:bound-NN}, uniformly in $r,B,$ for $\varepsilon >0,$
\begin{align*}
\E  | \xi (\eta '\cap B_{r}\cap B-x ) | ^{4+\varepsilon  }\leqslant &C \E[\exp (cR' )^{ {(4+\varepsilon )u} }] 
\end{align*}
and this quantity is finite if $\varepsilon $ is chosen such that $(4+\varepsilon )u<d$, and \eqref{eq:classic-stab-score}-\eqref{eq:moment-stab} hold, hence \eqref{eq:ultimate-assumption-bis} holds.

Let us check Assumption \ref{ass:nondeg}. Note that every result giving variance lower bounds for such functionals requires some kind of non-triviality check, as in \cite[Lemma 6.3]{PenYuk}, and the following result could likely be deduced from it. We prefer to present a self-contained proof since this example is supposed to illustrate the current method. Let  $A\subset (\mathbb{R}^{d})^{k}$ be such that $\tilde \varphi_{k}>0 $ on $A$ and $A\subset \text{\rm{int}}(\tilde Q_{\rho }^{k})$ for some $\rho >1$.  Hence there is $c>0$ such that for $\delta >\rho ,$ $p:=\P ( | F_{Q_{\delta }}(\eta ^{\rho }) | \geqslant c, | \eta ^{\rho } | =k+1)>0$ does not depend on $\delta .$
 It is clear that for $\gamma >3\rho  ,x\in \eta ^{\rho }$, if $ | \eta ^{\rho } | =k+1$, $N\hspace{-.5mm}N_{\leqslant k}(x;\eta ^{\rho }\cup \eta _{\gamma })\subset  \eta ^{\rho }.$ Reciprocally, if for   $x\in \eta _{\gamma }$, $ | B(x,  \|x\| -\rho )  \cap \eta _{\gamma } |> k+1 $, $x$  has its $k$ nearest neighbours in $\eta _{\gamma }$, and hence none in $\eta ^{\rho }$. If the two latter conditions are satisfied, $F_{Q_{\delta }}(\eta ^{\rho }\cup \eta _{\gamma })=F_{Q_{\delta }}(\eta ^{\rho })+F_{Q_{\delta }}(\eta _{\gamma })$, where $\delta >\gamma $. Hence
\begin{align*}
\P ( | F_{Q_{\delta }}(\eta _{\gamma })&-F_{Q_{\delta }}(\eta _{\gamma }\cup \eta ^{\rho }) |\geqslant c )\geqslant \P ( | F_{Q_{\delta }}(\eta _{\gamma })-F_{Q_{\delta }}(\eta _{\gamma }\cup \eta ^{\rho }) |\geqslant c, | \eta ^{\rho } | =k+1) \\
\geqslant& \P ( | F_{Q_{\delta  }}(\eta ^{\rho }) | \geqslant c, | \eta ^{\rho } | =k+1)-\sum_{j\in Q_{\delta }\setminus Q_{\gamma }}\P ( | \eta \cap B(j, \|j\|-\rho +\sqrt{d}) | \leqslant k)\\
\geqslant & p-\sum_{m=\gamma }^{\infty  }\kappa m^{d-1} C_{k}(m-\rho + \sqrt{d})^{k}\exp(-\kappa (m -\rho + \sqrt{d})^{d}).
\end{align*}
For $\gamma>3\rho  $ sufficiently large (and any $\delta >\gamma $), the last term is smaller than $p/2$, hence Assumption \ref{ass:nondeg} is satisfied.
\end{proof}
 \end{example}

 \subsection{Further applications and perspectives}
 
 An important part of the paper is devoted to shot noise excursions, but the results should apply also to most stabilizing models studied in the literature (packing functionals, Voronoi tessellation, boolean models, proximity graphs), see the example of statistics on nearest neighbours graphs above. 
 
 In some models, the independent marking is replaced by  {\it geostatistical marking}, also called dependent marking or external marking:
 let $m(x;\eta '),x\in \mathbb{R}^{d}$ be a random field measurable with respect to an independent homogeneous Poisson process $\eta '$ on $\mathbb{R}^{d}$, and consider the marked process $\{(x,m(x,\eta ')),x\in \eta \}$ instead of the independently marked process. Such a refinement is necessary to model a variety of random phenomena, such as gauge measurements for rainfalls or  tree sizes in a sparse forest, see \cite{SRD} and references therein.  Labelling the points of $\eta $ and $\eta '$ with two different colors yields that $\eta \cup \eta '$ has the law of an independently marked Poisson process, hence our results could be applied to appropriate statistics.
  
In the non-marked setting ($\M $ is a singleton), let $a>0$ be a scaling parameter, and consider  the random field $X=(X_{k})_{k\in \mathbb{Z} ^{d}}$, where $X_{k}=  \mathbf{1}_{\{ a\eta \cap (k+[0,1)^{d})=\emptyset \}  },k\in \mathbb{Z} ^{d}$. $X$ is an \emph {independent spin-model} where the parameter $p=\P(X_{0}=1)=\exp(-a)$ can take any prescribed value. Then all the previous results can be applied to functionals of the form 
\begin{align*}
F_{W}(X)=\sum_{k\in \mathbb{Z} ^{d}}F_{0}(X\cap {W}-k)\text{\rm{ or }}F'_{W}(X)=\sum_{k\in \mathbb{Z} ^{d}}F_{0}(X-k),
\end{align*} 
where $F_{0}$ is some functional on the class of subsets of $\mathbb{Z} ^{d}$, with finite second moment under iid Bernoulli input. Stabilising functionals and excursions functionals yield possible  applications, our findings might apply for instance to the results of \cite{RVY}, where more general classes of discrete input than Bernoulli processes are also treated. Seeing $F_{W}$  (or $F_{W}'$) as a functional of $\eta $, the variance and asymptotic normality results of Theorems \ref{thm:ultimate}-\ref{thm:ultimate-2} apply to $F_{W}$ under conditions of the type 
\begin{align*}
(\E\left|
F_{W}(X'\cap B)-F_{W}(X'\cap B\cap B_{r})
\right|^{4})^{1/4}\leqslant C_{0}(1+r)^{-\alpha },
\end{align*}where $B,B_{r}$ are like in \eqref{thm:ultimate}, and $X'$ is obtained from $X$ by forcing up to 2 spins $X_{k},X_{k'}$ to the value $1$  (the bound has to be uniform over $k,k'\in \mathbb{Z} ^{d}$).

 \section{Moment asymptotics}
 
 In this section, we give   results for second and fourth moments of a geometric functional under general conditions of non-triviality and polynomial decay.    The fourth order moment is useful for establishing Berry-Esseen bounds in the next section. The greek letter $\kappa $ still denotes a constant depending on $d,q,\alpha ,a_{-},a_{+}$ whose value may change from line to line.
    
 \begin{theorem}
 \label{thm-variance}Let $ \alpha > d ,W\subset \mathbb{Z} ^{d}$, $C_{0}\geqslant 0$. Let  $F_{0}\in \F$. \\
  Assume {\bf (i)} that for $k\in W$, { $G_{k}^{W}=F_{k}^{W},$ (resp. {\bf (i')} for $k\in \mathbb{Z} ^{d},G_{k}^{W}=F_{k}$) and let $G_{W}=\sum_{k\in W}G_{k}^{W}=F_{W}\;(\text{\rm{resp. }}G_{W}=F_{W}')$} as defined in \eqref{FW-general-form-Fk} (resp. \eqref{FW-general-form}), and   for all $r\geqslant 0 ,  B\in \mathcal{B}_{ W}^{r}\cup \{\mathbb{R}^{d}\}$,  
  \begin{align}
  \label{ass:upper-variance-2}
  \left(
\E 
   \left|F_{0}(\eta  \cap B_{r}\cap B )-F_{0}(\eta\cap B   )
\right| ^{2}
\right)^{1/2}
 \leqslant C_{0}(1+r)^{-\alpha } 
\end{align}  
(resp.     for all $r\geqslant 0  $,  
  \begin{align}
  \label{ass:upper-variance}
  \left(
\E 
   \left|F_{0}(\eta  \cap B_{r})-F_{0}(\eta   )
\right| ^{2}
\right)^{1/2}
 \leqslant C_{0}(1+r)^{-\alpha }). 
\end{align}  Then  for $k,j\in W$ (resp. $k,j\in \mathbb{Z} ^{d}$),
\begin{align}
\label{eq:cov-bound}
 \cov(G_{j}^{W},G_{k}^{W}) &\leqslant \kappa C_{0}^{2}   (1+\|k-j\|) ^{ -\alpha }, \\
\notag\sigma _{0}^{2}: =\sum_{k\in \mathbb{Z} ^{d}}\cov(F_{0} ,F_{k}) &<\infty ,
\end{align}
and $
\sigma _{0} 
>0$ if also Assumption \ref{ass:nondeg} holds.
If $W$ is bounded and non-empty,
\begin{align}
\label{eq:variance-exact-asymp}
\left|
  | W | ^{-1}{\var(G_{W})} -\sigma _{0}^{2}
\right|&\leqslant \kappa  C_{0}^{2 }   ( | \partialZ W | / | W | )^{1-d /\alpha }.
\end{align}   If furthermore $\alpha >2d$  
\begin{align}
\label{eq:moments-upper-bound}& \E\left(
G_{W}-\E G_{W}
\right)^{4}\leqslant \kappa C_{0}(\E (F_{0}-\E F_{0})^{4})^{3/4} | W | ^{2}.
\end{align}

\end{theorem}
 
The proof is deferred to Section \ref{sec:proof-variance}.

\section{Asymptotic normality}
  We give bounds to the normal in terms of Kolmogorov distance, defined in \eqref{eq:dK-def}, or Wasserstein distance, defined between two random variables $U,V$ as 
\begin{align*}
\\d_\W  (U,V)=\sup_{h\in \text{\rm{Lip}}_{1}} | \E [h(U)]-\E[h(V) ]| ,
\end{align*}where $\text{\rm{Lip}}_{1}$ is the set of $1$-Lipschitz functions $h:\mathbb{R}\to \mathbb{R}$.

\subsection{Malliavin derivatives}
\label{sec:malliavin}

It has been shown in different frameworks \cite{EstLeo,LRP,LSY,LPS} that, through inequalities called \emph {second-order Poincar\'e type} inequalities, Gaussian fluctuations of real functionals can be controlled by some second order difference operators defined on the random input. In the Poisson setting, this operator is incarnated by the  Malliavin derivatives. We define it here as it is a central tool  in the theory backing our results: for any  functional $F\in \F,\zeta \in \N$, and $\x\in \overline{\mathbb{R}^{d}}$, define the first order Malliavin derivative  $D_{\x}F\in \F$ by
\begin{align*}
D_{\x}F(\zeta   )=F(\zeta  \cup \{\x\})-F(\zeta   ),
\end{align*} and for $\x,\y\in \overline{\mathbb{R}^{d}},\zeta  \in \N, F\in \F_{0}$, the second order Malliavin derivative  is
\begin{align*}
D^{2}_{\x,\y}F(\zeta   )=& D_{\x}(D_{\y}F(\zeta   ))=F(\zeta   \cup \{\x,\y\})-F(\zeta \cup \{\x\})-F(\zeta \cup \{\y\})+F(\zeta   ).
\end{align*} 
 
One can use this object to quantify the spatial dependency of the functional $F$: a point $\y\in \overline{\mathbb{R}^{d}} $ has a weak influence on a point $\x\in \overline{\mathbb{R}^{d}}$ for the functional $F$ if its presence hardly affects the contribution of $\x$, i.e. $D_{\x}F(\eta )\approx D_{\x}F(\eta \cup \{y\})$, or in other words $D^{2}_{\x,\y}F(\eta )=D_{\y}(D_{\x}F(\eta ))\approx 0.$
The   proof of the following theorem is based on the result of Last, Peccati and Schulte \cite{LPS}, that  asserts that the functional $F_{W}$ exhibits Gaussian behavior as $W\to \mathbb{R}^{d}$,  as soon as $D_{\x,\y}F_{W}$ is small when $\x,\y$ are far away, uniformly in $W$. The speed of decay actually yields a bound on the speed of convergence  of $F_{W}$ towards the normal.

\begin{theorem}
\label{thm:BE}Let $W\subset \mathbb{Z} ^{d} $ {bounded}.
 Let $G_{W}\in \{F_{W} ,F_{W}'\}$ as defined in  \eqref{FW-general-form-Fk}-\eqref{FW-general-form}, with $F_{0}\in \F$, and let $M,M' \sim \mu $ independent.  Assume that for some $C_{0}>0, $ either  {\bf (i)} $G_{W}=F_{W}$ and for some $\alpha >2d, $ for all $k\in W,a.a. \,x\in \tilde W ,  a.a.\, y\in  \mathbb{R}^{d}$, $\eta' \in \{\eta ,\eta \cup \{(y,M')\} \},$ 
\begin{align}
\label{eq:moment-derivative-2}
  \left[ \E
 | D_{(x,M)}F_{0}((\eta'\cap  \tilde W)-k  ) | ^{4}
\right]^{1/4}
 \leqslant C_{0}(1+\|x\|)^{-\alpha  },x\in \mathbb{R}^{d},
\end{align}
or
 {\bf (i')} $G_{W}=F_{W}' $ and for some $\alpha >5d/2$, for a.a. $ x,y\in  \mathbb{R}^{d}$, $\eta' \in \{\eta ,\eta \cup \{(y,M')\} \},$ 
\begin{align}
\label{eq:moment-derivative}
  \left[ \E
 | D_{(x,M)}F_{0}(\eta'  ) | ^{4}
\right]^{1/4}
 \leqslant C_{0}(1+\|x\|)^{-\alpha   },x\in \mathbb{R}^{d}.
\end{align} 
 Then,   $ \sigma ^{2}:=\Var(G_{W})<\infty ,$ and if $\sigma > 0$, with  $ \tilde G_{W}=\sigma ^{-1}(G_{W}-\E G_{W})$, 
\begin{align}
\label{eq:BEbound-W-2}
 d_\W  (\tilde G_{W },N)\leqslant &\kappa \left(
 {C_{0}^{2}}{\sigma ^{- 2}}\sqrt{ | W | }+ {C_{0}^{3 }}{\sigma ^{ -3}}  | W | 
\right) \left(
 1+ 
\left(
 \frac{| \partialZ W |}{ | W |}  
\right)^{a}
\right),
\end{align}where $a=0$ in case {\bf (i)}, and $a=2( \alpha/d -2)$ in case {\bf (i')}.
Let $v :=\sup_{W} (G_{W}-\E G_{W})^{4} | W | ^{-2}\in \mathbb{R}_{+}\cup \{\infty \} $, then
\begin{align}
\label{eq:BEbound-K}
 d_\K  (\tilde G_{W },N)\leqslant &\kappa \left(
C_{0}^{2}\sigma ^{-2}\sqrt{ | W | }+C_{0}^{3}\sigma ^{-3} | W | +v^{1/4}C_{0}^{3}\sigma ^{-4} | W |^{3/2}
\right)\left(
 1+ 
\left(
 \frac{| \partialZ W |}{ | W |}  
\right)^{a}
\right).
\end{align}
Recall that   \eqref{ass:upper-variance-2} (or \eqref{ass:upper-variance} in case {\bf (i')}) is a sufficient condition for $v<\infty $.
\end{theorem} 

The proof is at Section \ref{sec:BE}

 \subsection{Proof of Theorems \ref{thm:ultimate} and \ref{thm:ultimate-2}}
\label{sec:proof-main}
We prove  Theorem \ref{thm:ultimate}  (resp. Theorem \ref{thm:ultimate-2}) using Theorems \ref{thm-variance} and \ref{thm:BE}.

Let $n\geqslant 1$ be such that $W=W_{n}$ is bounded and non-empty, $G_{W}=F_{W}$ (resp. $G_{W}=F'_{W}$), $\sigma ^{2}=\var(G_{W})$.  Assumption \eqref{eq:ultimate-assumption-bis} (resp. \eqref{eq:ultimate-assumption}) clearly implies \eqref{ass:upper-variance-2} (resp. \eqref{ass:upper-variance}), and therefore  \eqref{eq:variance-exact-asymp} holds:\begin{align*} 
\left|
  | W | ^{-1}{\sigma ^{2}} -\sigma _{0}^{2}
\right|&\leqslant \kappa  C_{0}^{2 }   ( | \partialZ W | / | W | )^{1-d /\alpha }.
\end{align*} Let $y\in \mathbb{R}^{d},k \in W,x\in \tilde W-k,\x=(x,M), \eta' \in \{\eta ,\eta \cup \{(y,M')\}\}$  as in  \eqref{eq:moment-derivative-2} (resp. \eqref{eq:moment-derivative}), $\eta ''=\eta '\cup \{\x\},B=\tilde W-k\text{\rm{ (resp. }}B=\mathbb{R}^{d}),r= \|x\| / a_{+}.$ Note that $x\in B\setminus  B_{r}$, hence
\begin{align*}
 D_{\x}F_{0}(\eta '\cap B)   =&   F_{0}((\eta '\cap B)\cup \{\x\})-F_{0}(\eta '\cap B)   \\
 =&F_{0}((\eta '\cup \{\x\})\cap B)-F_{0}(\eta '\cap B)\\
 =&F_{0}((\eta '\cup \{\x\})\cap B)-F_{0}((\eta '\cup \{\x\})\cap B\cap B_{r})+F_{0}((\eta '\cup \{\x\})\cap B\cap B_{r})-F_{0}(\eta '\cap B)\\
 =&F_{0}(\eta ''\cap B)-F_{0}(\eta ''\cap B\cap B_{r})+F_{0}(\eta '\cap B\cap B_{r})-F_{0}(\eta '\cap B).
\end{align*}
Applying  \eqref{eq:ultimate-assumption-bis} (resp. \eqref{eq:ultimate-assumption}) twice  with $x_{1}=x,x_{2}=y $ yields
\begin{align*}
\left(
\E | D_{\x}F_{0}(\eta '\cap B) | ^{4}
\right)^{1/4}\leqslant C_{0}(1+r)^{-\alpha},
\end{align*}
 hence \eqref{eq:moment-derivative-2}  (resp. \eqref{eq:moment-derivative}) holds, and \eqref{eq:BEbound-W-2} holds. Since furthermore Assumption \ref{ass:nondeg} holds, Theorem \ref{thm-variance}  yields $\sigma _{0}>0$, and for $n$ sufficiently large,
 $
\sigma ^{-2}\leqslant  2| W |^{-1}  \sigma _{0}^{-2}
 ,$ hence, with $\tilde G_{W}:=(G_{W}-\E G_{W})(\var G_{W})^{-1/2}$, for $n$ sufficiently large,  using also \eqref{eq:ball-like},
\begin{align*}
d_{\W}(\tilde G_{W},N)\leqslant \kappa | W | ^{-1/2}\left(
{C_{0}^{2}}{\sigma _{0}^{-2}} +{C_{0}^{3}}{\sigma _{0}^{-3}}
\right).
\end{align*} Since  \eqref{eq:ultimate-assumption-bis} (resp. \eqref{eq:ultimate-assumption}) holds with $\alpha >2d$, we have furthermore by \eqref{eq:moments-upper-bound}:\begin{align*}
 &v=\limsup_{n\geqslant 1}  \E\left(
G_{W_{n}}-\E G_{W_{n}}
\right)^{4}/ | W _{n}| ^{2}\leqslant \kappa C_{0}(\E (F_{0}-\E F_{0})^{4})^{3/4} .
\end{align*}  Applying \eqref{eq:ultimate-assumption-bis} with $r=0,B=\mathbb{R}^{d}$ gives $\E (F_{0}-\E F_{0})^{4}\leqslant \kappa C_{0}^{4}$. The bound on Kolmogorov distance \eqref{eq:ultimate-dK} (resp. \eqref{eq:ultimate-dK-2}) follows easily by \eqref{eq:BEbound-K}.

It remains to prove that   $  G'_{W} :=(\sigma _{0}^{2} | W | )^{-1/2}( G_{W} -\E G_{W} )$ follows a central limit theorem. We achieve it by proving that its Wasserstein distance to the normal goes to $0$. The triangular inequality yields
\begin{align*}
d_{{\W}}(  G_{W}' ,N)\leqslant &\E\left|
 G'_{W}-\tilde G_{W}
\right|+d_{{\W}}(\tilde G_{W},N) \\
\leqslant &\left|
\frac{1}{\sigma _{0} \sqrt{ | W |} }-\frac{1}{\sqrt{\var(G_{W})}}
\right|\E\left|
G_{W}-\E G_{W}
\right| +d_{{\W}}(\tilde G_{W},N)\end{align*}
which indeed goes to $0$ by \eqref{eq:variance-exact-asymp}.

\section{Application to shot-noise processes}
\label{sec:SN}
    Let the notation of the introduction prevail. For the process $f_{\eta }$ (see  \eqref{def-SN}) to be well defined,  assume throughout the section that for some $\tau  >0$,
\begin{align}
\label{eq:exist-SN}
\int_{\M}\int_{B(0,\tau )^{c}} | g_{m}(x) | dx\mu (dm)<\infty ,
\end{align} and let $\N_{0}$ be the class of locally finite  $\zeta $ such that $\sum_{(x,m)\in \zeta  } | g_{m}(x) |<\infty ,x\in \mathbb{R}^{d}. $ The fact that $\eta \in \N_{0}$ a.s. follows from the Campbell-Mecke formula.

We study in this section the behaviour of functionals of the excursion set $\{f_{\eta }\geqslant u\},u\geqslant 0$. We use the general framework of {random measurable sets}. A {\it random measurable set} is a random variable taking values in the space $\mathcal{M}_d $ of measurable subsets of $\mathbb{R}^{d}$, endowed with  the Borel $\sigma $-algebra $\mathcal{B}(\mathcal{M}_d )$ induced by the local convergence in measure, see Section 2 in \cite{GalLac}. Regarding the more familiar setup of random closed sets,   in virtue of Proposition 2 in \cite{GalLac},    a random measurable set
which realisations are a.s. closed can be assimilated to a random closed set. 
\renewcommand{\L}{\mathcal{L}}

 \subsection{Volume of excursions}
\label{sec:SN-volume}

For $u\in \mathbb{R}$ fixed$,W\subset \mathbb{Z} ^{d},\zeta \in \N$, define
\begin{align*}
F_{W}(\zeta )=\ell^d(\{f_{\zeta \cap \tilde W}\geqslant u\}\cap \tilde W),\;\;
F'_{W}(\zeta )=\ell^d(\{f_{\zeta }\geqslant u\}\cap \tilde W).
\end{align*} 
A central limit theorem  for the volume of a certain family of shot noise excursions has been derived in \cite{BST}, under the assumption that $f_{\eta  }( 0  )$ has a uniformly bounded density and $\int|  g_{m}(x) |\mu (dm) $ decreases  sufficiently fast as $\|x\|\to\infty $, using the associativity properties of non-negative shot-noise fields. In some specific cases, the bounded density can be checked manually with computations involving the Fourier transform. In this section, we refine this result in several ways:
\begin{itemize}
\item A general model of random function is treated, it can in particular take negative values, allowing for compensation mechanisms (see \cite{Lan02}). For $u >0$, to avoid trivial cases we assume
\begin{align}
\label{eq:non-trivial}
\mu (\{m\in \M :g_{m} \geqslant 0  \})\neq 0.
\end{align}
\item The precise variance asymptotics are derived.
\item Weaker conditions are required for the results to hold, in particular  bounded density is not needed.
\item The likely optimal rate of convergence in Kolmogorov distance towards the normal is given.
\item Boundary effects under finite input are considered, in the sense that only points falling in a bounded window (growing to infinity) contribute to the field. The case of infinite input is also treated.
\end{itemize}
The application to shot noise excursions is a nice illustration of the versatility of the general method derived in this article. We give examples of fields with no marginal density to which the results apply, such as  sums of indicator functions, or   of kernels with a singularity in $0$. Controlling the density of shot-noise fields is in general crucial for deriving results on fixed-level excursions.   The case of indicator kernels is treated in Section \ref{sec:sn-general}.

%

\begin{assumption}
\label{ass:SN-g}
Let $f_{\eta }$ be of the form \eqref{eq:sn-g} with $g$ such that $ | g(x) | \leqslant c\|x\|^{-\lambda },\|x\|\geqslant 1$ for some $\lambda >11d,c>0$. Assume that there is $\varepsilon >0,c>0$ such that  
\begin{align}
\label{eq:tail-cond-SN-volume}
 \int_{0}^{ r}\frac{\rho ^{-2}\wedge \rho ^{2(d-1)}}{-g'(\rho )}d\rho \leqslant c\exp(c r^{d-\varepsilon }),r>0.
\end{align}
\end{assumption}

Lemma \ref{lm:sn-density} below yields that if $f_{\eta }$ satisfies this assumption, we can somehow control its density: for   $a\in (0,1)$    there is $ c_{a}>0$ such that  
\begin{align}
\label{eq:density}
 \sup_{v\in \mathbb{R}  , \delta >0 } \P(f_{\eta  }(0   )\in (v-\delta ,v+\delta )  )\leqslant c_{a}\delta ^{a}. 
\end{align}    This result might be of independent interest, and is proved after Lemma \ref{lm:sn-density}. 
 Here are examples of functions fulfilling Assumption \ref{ass:SN-g} (and hence satisfying \eqref{eq:density}), note that nothing prevents $g$ from having a singularity in $0$.

\begin{example}
\label{ex:sn}
Theorem \ref{thm:SN-vol} below applies in any dimension to $g(\rho )=C  \rho ^{-\nu    }\mathbf{1}_{\{\rho \leqslant 1\}}+ g_{1}(\rho )\mathbf{1}_{\{\rho >1\}},\rho >0$  and $g_{1}(\rho )$ is for instance of the form $\exp(-a \rho ^{\gamma })$ or $\rho ^{-\lambda   }$, with $a, \nu  >0,\lambda >11d,\gamma <d,C>0$. Such fields don't necessarily have a finite first-order moment, and  are used for instance in \cite{BaccelliBiswas} to approximate stable fields, or  for modeling telecommunication networks.   
\end{example}

To give results in the case where boundary effects are considered, we need an additional hypothesis on the geometry of the underlying family of windows $\W=\{W_{n};n\geqslant 1\}$. For $\theta >0$, let $\mathcal{C}_{\theta }$ be the family of cones $C\subset \mathbb{R}^{d}$ with apex $0$ and aperture $\theta $, i.e. such that $\mathcal{H}^{d-1}(C\cap \mathcal{S}^{d-1})\geqslant \theta $. Let $\mathcal{C}_{\theta ,R}=\{C\cap B(0,R):C\in \mathcal{C}_{\theta }\}$ for $R\geqslant 0.$ Say that $\W$ has aperture $\theta >0$ if for all $W\in \W$ with diameter $r>0$, $W$ has {\it aperture $\theta$ }: for $x\in \tilde W$,   there is   $ C\in \mathcal{C}_{\theta,\ln(r)^{1/2d}  }$ such that 
 $
 (x+C)\subset \tilde W.$

%
%
  \begin{theorem}
\label{thm:SN-vol}
Let $u> 0$. Let $G_{W}= F'_{W} $, or 
$G_{W}=F_{W}$ if $\W$ is assumed to have aperture $\theta>0 $.
Assume   that    Assumption \ref{ass:SN-g}
holds. Then as $ | \partialZ W | / | W | \to 0$, $\var(G_{W})\sim \sigma _{0}^{2} | W | $, $(G_{W}-\E G_{W})(\sigma _{0}\sqrt{ | W| })^{-1}$ satisfies a central limit theorem, with
\begin{align}
\label{eq:sigma0-volume}
\sigma _{0}^{2}=\int_{\mathbb{R}^{d}}\left[
\P(f_{\eta }(0 )\geqslant u,f_{\eta }(x)\geqslant u)-\P(f_{\eta }(0 )\geqslant u)^{2}
\right]dx>0.
\end{align}
Also,   the convergence rate \eqref{eq:BEbound-W-2}   in Kolmogorov distance  holds for  $\tilde G_{W}$.
\end{theorem}
  This result requires $f$ to be under the form \eqref{eq:sn-g} mainly because of the density estimates provided by Lemma \ref{lm:sn-density}, but under general density assumptions, it could apply to more general models of the form \eqref{def-SN}.
 Let us state a lemma that will be required in the proof, and in other results concerning the non-triviality of shot-noise excursions.
 
\begin{lemma}
\label{lm:nontrivial-levelset} Let $f_{\eta }$ be of the form \eqref{def-SN}.
Assume that
\begin{align}
\label{eq:ass-nontrivial-volume}
\text{\rm{for some }}M\subset \{m\in \M:\ell^d(g_{m}^{-1}((0,\infty )))>0\},\;\;\;\mu (M)>0.
\end{align}  
Then there is   $\rho  >1$ such that  for $\beta \geqslant 1 $, $
\E(\ell^d(\{f_{\eta ^{\rho }}> u\}\cap \tilde Q_{\beta   }))>0.
$
\end{lemma}

\begin{proof} 
Basic measure theory yields $\varepsilon >0,\rho  >1$ such that 
\begin{align*}
\mu (\{m:\ell^d(g_{m}^{-1}((\varepsilon ,\infty ))\cap \tilde Q_{\rho -1})>0\})>0.
\end{align*}Let $t\in \tilde Q_{1},k>u/\varepsilon $, and $X_{i}=(Y_{i},M_{i}),i\leqslant k$ iid couples of $\tilde Q_{\rho }\times \M$,  
  and $U_{i}:=g_{M_{i}}(t-Y_{i})$. We have $\P(U_{1}\geqslant \varepsilon )>0$, hence 
\begin{align*}
\P(f_{ \eta ^{\rho }}(t)\geqslant u)\geqslant &\P(f_{ \eta ^{\rho }}(t)\geqslant k\varepsilon )
\geqslant\P ( | \eta ^{\rho }\cap \tilde Q_{\rho } | =k) \P(U_{1}\geqslant \varepsilon )^{k} >0.
\end{align*}
Then Fubini's theorem yields for $\beta \geqslant 1 $
\begin{align*}
\E\ell^d&(\{t\in \tilde Q_{\beta   }:f_{  \eta ^{\rho } }(t)\geqslant u\})\\
\geqslant& \E\ell^d(\{t\in \tilde Q_{1}:f_{  \{X_{1},\dots ,X_{k}\}}(t)\geqslant u\})\P ( | \eta  ^{\rho } | =k )>0.
\end{align*}\end{proof} 

\begin{proof}[Proof of Theorem  \ref{thm:SN-vol}]

The decay assumption on $g$ yields that \eqref{eq:exist-SN} holds for  $\tau =1,$ and the left hand member of \eqref{eq:ultimate-assumption} is uniformly bounded for $r\leqslant2 \sqrt{d} $. From now on we take $r> 2 \sqrt{d} .$  We wish to prove the the conditions of Theorems \ref{thm:ultimate} and \ref{thm:ultimate-2} are satisfied with the functional $F_{0}(\zeta )=\int_{\tilde Q_{1}}\mathbf{1}_{\{f_{\zeta }(t)\geqslant u\}}dt.$  
Let us start by proving Assumption \ref{ass:nondeg}. 
Let $M=\{m\in \M:g_{m}\geqslant 0\} $. For $\rho >0,$ let $\Gamma _{\rho }$ be the event that $\eta ^{\rho }\subset \tilde Q_{\rho }\times  M$ (i.e. all functions of $\eta ^{\rho }$ are non-negative). Since $\mu (M)>0$, $\P (\Gamma _{\rho })>0$.
 Lemma  \ref{lm:nontrivial-levelset} yields $p>0, \rho  >1$ such that for $t\in \tilde Q_{1 },$
 $
\P (f_{\eta ^{\rho }}(t)>2u| \Gamma _{\rho })\geqslant p$. Also $\E  | f_{\eta _\gamma }(t) | \to 0$ as $\gamma \to \infty $  uniformly in $t\in \tilde Q_{1},$ hence for $\gamma $ sufficiently large, $t\in \tilde Q_{1},
\P ( | f_{\eta _\gamma }(t) | <u )>\frac{1}{2}.$ Conditionaly on $\Gamma _{\rho }$, $f_{\eta _{\gamma }\cup \eta ^{\rho }}=f_{\eta ^{\rho }}+f_{\eta _{\gamma }}\geqslant f_{\eta _{\gamma }}$.
Hence, for $\delta >\gamma >\rho $,  
\begin{align*}
\mathbf{1}_{\{\Gamma _{\rho }\}}\left|
F_{Q_{\delta }}(\eta _{\gamma }\cup \eta ^{\rho })-F_{Q_{\delta }}(\eta _{\gamma })
\right|=\mathbf{1}_{\{\Gamma _{\rho }\}} \int_{\tilde Q_{\delta }}\mathbf{1}_{\{f_{\eta ^{\rho }\cup \eta _{\gamma }}(t)>u ,f_{ \eta _{\gamma }}(t)<u\}}dt
&\geqslant \mathbf{1}_{\{\Gamma _{\rho }\}}G\\
 \text{\rm{where }} G:= \int_{\tilde Q_{1 }}\mathbf{1}_{\{ | f_{\eta _{\gamma }} (t)| <u,f_{\eta ^{\rho }}(t)>2u\}}dt .\\
\E [G | \Gamma _{\rho }]\geqslant \int_{\tilde Q_{1  }}\P (  f_{\eta ^{\rho }}(t)  >2u | \Gamma _{\rho })\P ( | f_{\eta  _{\gamma } }(t) | <u)dt&\geqslant \frac{p }{2}.
\end{align*}Since $G\leqslant 1$,  $\P (G\geqslant p/4 | \Gamma _{\rho })\geqslant p/4 >0$. Hence $\P ( | F_{Q_{\delta }}(\eta _{\gamma }\cup \eta ^{\rho })-F_{Q_{\delta }}(\eta ^{\rho }) | >c)\geqslant \P (\Gamma _{\rho }) \P (G\geqslant p/4 | \Gamma _{\rho }) =:p'>0$, hence 
 Assumption \ref{ass:nondeg} is satisfied.

Let us now prove that  \eqref{eq:ultimate-assumption-bis} holds in the case $G_{W}=F_{W} $ (or \eqref{eq:ultimate-assumption} in the case $G_{W}=F_{W}'$).
Let $x_{1},x_{2}\in \mathbb{R}^{d},$ $M_{1},M_{2}$ independent marks with law $\mu $, $r \geqslant 0,$ $\zeta \subset \{(x_{1},M_{1}),(x_{2},M_{2})\},\eta '=\eta \cup \zeta $. Let $B=\mathbb{R}^{d}$ in the case of infinite input ($G_{W}=F_{W}'$), and let $ B \in \mathcal{B}_{\W}^{r}$ otherwise ($G_{W}=F_{W}$). 
Jensen's inequality yields
\begin{align*}
\left|
F_{0}(\eta '  \cap B)-F_{0}(\eta' \cap  B_{r}\cap B)
\right|^{4}=&\left[
\int_{\tilde Q_1}\left(
\mathbf{1}_{\{f_{\eta '\cap B}(t)\geqslant u\}}-\mathbf{1}_{\{f_{\eta '\cap B_{r}\cap B}(t)\geqslant u\}}
\right)dt
\right]^{4}\\
\leqslant &\int_{\tilde Q_1}\left|
\mathbf{1}_{\{f_{\eta '\cap B}(t)\geqslant u\}}-\mathbf{1}_{\{f_{\eta '\cap B_{r}\cap B}(t)\geqslant u\}}
\right|dt ,
\end{align*}
 and for $t\in \tilde Q_1,r>2\sqrt{d} ,$
\begin{align*}
\left|
f_{\eta '\cap B}(t)-f_{\eta '\cap B_{r}\cap B}(t)
\right|=&\left|
\sum_{\x=(x,m)\in (\eta '\cap B) \setminus  B_{r}}
g_{m}(t-x)\right|   \\
\leqslant &\delta _{r,t }:=\sum_{\x=(x,m)\in \eta'\setminus  B(t,a_{-}(r-\sqrt{d}))} | g_{m}(t-x) |.
\end{align*}Note that $\delta _{r,t}$ is independent from $\eta \cap B(t,a_{-}r/2)$ and its law does not depend on $t\in \tilde Q_1$. Since $B=\tilde Z$ for some $Z\subset \mathbb{Z} ^{d}$ and $0\in B, t\in B$. Let $ R=1\wedge \ln(a_{-}r)^{ \frac{1 }{d-\varepsilon/2 }}$, where $\varepsilon $ is from Assumption \ref{ass:SN-g}. Since $B$ intersects $B_{r}^{c}$,  it has diameter at least $a_{-}r$ and since $W$ has aperture $\theta $,   there is a solid cone $C_{t}\in \mathcal{C}_{\theta ,R}$ such that, with $D_{t}=(C_{t}+t) $, $D_{t}\subset B$. In the infinite input case, the latter trivially holds with $B=\mathbb{R}^{d},\theta =\sigma _{d-1}:=\mathcal{H}^{d-1}(\mathcal{S}^{d-1}),D_{t}=B(t,R)$.
We have 
\begin{align}
\notag\E\Big|
F_{0} (\eta '&\cap B) - F_{0}(\eta '\cap B_{r}\cap B)
\Big|^{4}\leqslant \sup_{t\in \tilde Q_1}\P(f_{\eta '\cap B}(t)\in [u-\delta_{r,t },u+\delta_{r ,t}])\\
\notag\leqslant &\sup_{t\in \tilde Q_1}\P\left(
f_{\eta \cap D_{t}}(t)+f_{\eta \cap (B\setminus   D_{t})}(t) +f_{\zeta \cap B}(t)\in [u-\delta _{r,t},u+\delta _{r,t}]\right)\\
\notag\leqslant &\sup_{t\in \tilde Q_1}\E(\P (f_{\eta \cap D_{t}}(t)\in [u-f_{\eta \cap (B\setminus D_{t})}(t)-f_{\zeta\cap B }(t)\pm \delta _{r,t}] \;|\; \sigma (\zeta ,\eta \cap (B\setminus D_{t}))))\\
\notag\leqslant &\sup_{t\in \tilde Q_1}\E\left(\sup_{v\in \mathbb{R}}\P (
f_{\eta \cap D_{t}}(t) \in [v-\delta _{r,t},v+\delta _{r,t}]\; |\; \sigma (\delta _{r,t}))
\right)\\
\label{eq:tructruc}\leqslant &\E\left(\sup_{C\in \mathcal{C}_{\theta ,R}}
\sup_{v\in \mathbb{R}}\left[
\P(f_{\eta \cap C}(0)\in [v-\delta _{r,0},v+\delta _{r,0}])\; |\; \sigma (\delta _{r,0}))
\right]
\right).
\end{align}
To bound this quantity we need to study the density of the shot-noise field.

\begin{lemma}
\label{lm:sn-density}Assume that $f_{\eta }$ is of the form \eqref{eq:sn-g}.
Let  $\delta >0,R\geqslant 1   $.
Then for $v\in \mathbb{R},C\in \mathcal{C}_{\theta ,R}$,
\begin{align*}
\P(f_{\eta \cap C}(0 )\in [v-\delta ,v+\delta ]\; , \; | \eta \cap C | \geqslant 2)\leqslant \kappa \delta 
 \int_{0}^{R}\frac{(\rho ^{-2}\wedge \rho ^{2(d-1)})d\rho }{- g'(\rho )}.
\end{align*}
\end{lemma}

Before proving this result, let us conclude the proof of Theorem  \ref{thm:SN-vol}. Assume without loss of generality $r\geqslant 2r_{0}/a_{-}$. By Assumption \ref{ass:SN-g}, \eqref{eq:tructruc} is bounded by 
 $
\sup_{C\in \mathcal{C}_{\theta ,R}} \P( | \eta \cap C | <2 ) +c\kappa (\E[ \delta _{r,0} \exp(cR^{d-\varepsilon })]).
 $
The decay assumption on $g$ yields that $\E (\delta _{r,0})\leqslant \kappa (1+r)^{-\lambda +d }$. Hence \eqref{eq:tructruc}  is bounded by 
\begin{align*}
  (1+\kappa   R^{d})\exp(-\kappa \theta R^{d})+c\kappa \exp(c\kappa \ln(r)^{\frac{d-\varepsilon }{d-\varepsilon /2}})(1+r)^{-\lambda +d }\leqslant \kappa (1+r)^{-(\lambda  '-d)}
\end{align*}for any $\lambda '\in (11d,\lambda  ).$ Hence \eqref{eq:ultimate-assumption-bis} and \eqref{eq:ultimate-assumption} hold  with $\alpha = (\lambda  '-d)/4>5d/2$.
\end{proof}

\begin{proof} [Proof of Lemma \ref{lm:sn-density}]
 Let $\lambda =\frac{\sigma _{d-1}}{\theta \kappa _{d} },n_{R}= | \eta \cap C| $ be the number of germs (Poisson variable with parameter $ \ell^d(C) = R^{d}/\lambda $), and let $g_{R}(x)=g(\|x\|) \mathbf{1}_{\{x\in C\}},$ so that $f_{\eta \cap C} (0 )=\sum_{i=1}^{n_{R}} g_{R}(X_{i})$ where the $X_{i}$ are uniform iid in $  C$. Call  $\mu _{R}$ the distribution of the $ g_{R}(X_{i})$. We have for  every $b>a\geqslant g(R),$ since $g$ is one-to-one and continuous
\begin{align*}
\mu _{R}([a,b))=&\frac{\lambda }{   R^{d}}\int_{ C } \mathbf{1}_{\{a\leqslant   g (\|x\|)<b\}}dx=\frac{\lambda }{  R^{d}}\int_{g^{-1}(b)}^{g^{-1}(a)}\theta \rho ^{d-1}d\rho \\
=&\frac{\sigma _{d-1} }{d\kappa _{d} R^{d}} (g  ^{-1}(a)^{d}-g  ^{-1}(b)^{d}),\\
\end{align*}
  whence $\mu _{R}$ has density
 $
\varphi _{R}(a)=\mathbf{1}_{\{a\geqslant g(R)\}} \frac{\sigma _{d-1}  }{\kappa _{d}R^{d}} \left(
\frac{g ^{-1}(a)^{d-1}}{-g  '(g  ^{-1}(a))}
\right) .$
 Then, denoting by $\varphi _{R}^{\otimes n}$ the density $\varphi _{R}$ convoluted with itself $n$ times on the real line, 
\begin{align}
\notag\P(f_{\eta \cap C}(0)\in [v-\delta ,v+\delta ] \; , \; | \eta \cap C | \geqslant 2)\leqslant& \sum_{n=2}^{\infty }\P(n_{R}=n)\P\left(
\sum_{i=1}^{n}g_{R}(X_{i})\in [v-\delta ,v+\delta ] 
\right)\\
\notag\leqslant &\sum_{n\geqslant 2}\P(n_{R}=n)\|\varphi _{R}^{\otimes n}\|_{\infty }2\delta \\
\label{eq:intermed-densite-SN-2}\leqslant&2\sup_{n\geqslant 2}\|\varphi _{R}^{\otimes n}\|_{\infty }\delta .
\end{align}
  Due to convolution properties, for $n\geqslant 2$, 
\begin{align*}
\|\varphi _{R}^{\otimes n}\|_{\infty }
\leqslant &\|\varphi _{R}^{\otimes 2}\|_{\infty }\leqslant  \int_{\mathbb{R}}\varphi _{R }^{2}(a)da=\E \varphi _{R}(g_{R}(X_{1}))\\
=&\frac{ \lambda }{ R^{d}}\int_{C}\varphi _{R }(g (\|x\| ))dx\\
\leqslant &\frac{ \lambda }{ R^{d}}\int_{ C}\frac{g ^{-1}(g(\|x\|))^{d-1}}{-g '(g ^{-1}(g(\|x\|)))}\frac{\sigma _{d-1} dx}{ \kappa _{d}R^{d}}\\
=&\left(
\frac{\sigma _{d-1} }{\kappa _{d}R^{d}}
\right)^{2}\int_{0}^{R}\frac{1}{-g '(\rho )}\rho ^{2(d-1)}d\rho \\
\leqslant &\left(
\frac{\sigma _{d-1} }{\kappa _{d} }
\right)^{2}\left( \frac{1}{R^{2d}}\int_{0}^{1 }\frac{\rho ^{2(d-1)}}{-g'(\rho )}d\rho +
\int_{1 }^{R}\frac{\rho ^{-2}d\rho }{-g '(\rho )}
\right),
\end{align*}
which concludes the lemma after reporting in \eqref{eq:intermed-densite-SN-2}.
\end{proof}

\begin{proof}[Proof of claim \eqref{eq:density}]
Let $v>0,\delta >0$. Let $R=R_{\delta }:=1\wedge  | \ln(\delta ) | ^{\frac{1}{d-\varepsilon /2}}$. Introduce the events $A_{\delta ,v}=\{f_{\eta ^{R}}(0)\in (v-\delta ,v+\delta )\}$, $A_{\delta ,v}'=\{f_{\eta }(0)\in (v-\delta ,v+\delta )\}$, $B_{\delta }=\{ | \eta ^{R} | \geqslant 2\}$. Since Assumption \ref{ass:SN-g} holds, Lemma \ref{lm:sn-density} yields $\P (A_{\delta,v } \cap B_{\delta })\leqslant c_{a}\delta ^{a}$ for all $v>0$. Let $U_{\delta }=f_{\eta }(0)-f_{\eta ^{R_{\delta }}}(0)$. Note that $U_{\delta }$ is independent from $f_{\eta ^{R_{\delta }}}(0)$. We have 
\begin{align*}
\P (A_{\delta,v })\leqslant& \P (A_{\delta ,v}\cap B_{\delta })+\P (B_{\delta }^{c})\leqslant c_{a}\delta ^{a}+o(\delta ^{a})\leqslant c_{a}'\delta ^{a}.\\
\P (A_{\delta ,v }')=&\E \left[
\P (f_{\eta ^{R}}(0)+U_{\delta }\in (v-\delta ,v+\delta ) | U_{\delta })\right]\\
\leqslant & \E (\P (A_{\delta ,v-U_{\delta }} | U_{\delta }))
\leqslant \E (c_{a}'\delta ^{a})=c_{a}'\delta ^{a}
,
\end{align*}
hence the claim is proved.
\end{proof}

\subsection{Perimeter}
\label{sec:SN-per}

We use in this section the variational definition of perimeter, following Ambrosio, Fusco and Pallara \cite{AFP}. Define the \emph{perimeter} of a measurable set $A\subset \mathbb{R}^{d}$ within $U\subset \mathbb{R}^{d}$ as the  total variation of its indicator function
\begin{align*}
\Per(A;U):=\sup_{\varphi \in \mathcal{C}_{c}^{1}(U,\mathbb{R}^{d}):\|\varphi \|\leqslant 1}\int_{\mathbb{R}^{d}}\mathbf{1}_{A}(x)\text{\rm{div}}\varphi (x)dx,
\end{align*}where $\C_{c}^{1}(U,\mathbb{R}^{d})$ is the set of continuously differentiable functions with compact support in $U$. Note that for regular sets, such as $\mathcal{C}^{1}$ manifolds, or convex sets with non-empty interior,   this notion meets the classical notion of $(d-1)$-dimensional Hausdorff surface measure \cite[Exercise 3.10]{AFP}, even though the term {\it perimeter} is traditionally  used for $2$-dimensional objects. It is a possibly infinite quantity, that might also have counterintuitive features for pathological sets (\cite[Example 3.53]{AFP}). The main difference with the traditional perimeter is that the variational one obviously cannot detect the points of the boundary whose neighborhoods don't charge the volume of the set, such as in line segments for instance.

 For any  measurable function $f:\mathbb{R}^{d}\to \mathbb{R}$ and level $u\in \mathbb{R}$, the perimeter of the excursion $\Per(\{f\geqslant u\};U)$ within $U$  is a well-defined  quantity. To be able to compute it efficiently, we must  make additional assumptions on the regularity of $f$. Following \cite{BieDes}, we assume that $f$ belongs to the space $BV(U)$ of functions with bounded variations, i.e.   $f\in L^{1}(U)$ and its  variation above $U$ is finite: 
\begin{align*}
V(f,U): =\sup_{\varphi \in \mathcal{C}_{c}^{1}(U,\mathbb{R}^{d}):\|\varphi \|\leqslant 1}\int_{U}f(x)\text{\rm{div}}\varphi (x)dx<\infty .
\end{align*}
The original (equivalent) definition  states that $f\in L^{1}(U)$ is in $ BV(U)$ if and only if the following holds (\cite[Proposition 3.6]{AFP}): there exists signed Radon measures  $D_{i}f$ on $U$$, 1\leqslant i\leqslant d$, called \emph{directional derivatives} of $f$, such that for all $\varphi \in \mathcal{C}_{c}^{\infty }(\mathbb{R}^{d})$,
\begin{align*}
\int_{U}f(x)\text{\rm{div}}\varphi (x)dx=-\sum_{i=1}^{d}\int_{U}\varphi _{i}(x)D_{i}f(dx).
\end{align*}Then there is a finite Radon measure $\|Df\|$ on $U$, called \emph{total variation measure}, and a   $\mathcal{S}^{d-1}$-valued function $\nu _{f}(x),x\in U$, such that $Df=\sum_{i}D_{i}f=\|Df\|\nu _{f}$. 
According to the Radon-Nikodym theorem, the total variation can be decomposed as 
\begin{align}
\label{eq:measure-decomp}
\|Df\|=\nabla f \ell^d+D^j f +D^c f
\end{align}where $\nabla f$ is defined as the density of the continuous part of $\|Df\|$ with respect to $\ell^d$, $D^c f+D^jf$ is the singular part of $\|Df\|$ with respect to Lebesgue measure, decomposed in the  \emph{Cantor part} $D^cf$, and the jump part $D^j f$, that we specify below, following \cite[Section 3.7]{AFP}. 

For $x\in U,$ denote by  $H_{x}$ the  affine hyperplane containing $x$ with outer  normal vector $\nu _{f}(x)$. For $r>0$, denote by $B^{+}(x,r)$ and $B^{-}(x,r)$ the two components of $B(x,r)\setminus H_{x}$, with $\nu _{f}(x)$ pointing towards $B^{+}(x,r)$. Say that $x$ is a \emph{regular point} if there are two values $f^{+}(x)\geqslant f^{-}(x)$ such that 
\begin{align}
\label{eq:representation-jump-measure}
\lim_{r\to 0}r^{-d}\hspace{-0.3cm}\int\limits_{B^{+}(x,r)} \hspace{-0.3cm}| f^{+}(x)-f(y) | dy=\lim_{r\to 0}r^{-d}\hspace{-0.3cm}\int\limits_{B^{-}(x,r)} \hspace{-0.3cm}| f(y)-f^{-}(x) | dy=0.
\end{align}It turns out that the set of non-regular points has $\mathcal{H}^{d-1}$-measure $0$ (\cite[Th. 3.77]{AFP}), and the set $J_{f}$ of points where $f^{+}(x)>f^{-}(x)$, called \emph{jump points},  has Lebesgue measure $0$  (\cite[Th. 3.83]{AFP}). Then, the jump measure of $f$ is represented by 
\begin{align*}
D^j f(dx)=\mathbf{1}_{\{x\in J_{f}\}}(f^{+}(x)-f^{-}(x))\mathcal{H}^{d-1}(dx),
\end{align*}
where $\mathcal{H}^{d-1}$ stands for the $(d-1)$-dimensional Hausdorff measure.

In the classical case where $f$ is continuously differentiable on $U$, $Df=\nabla f\ell^d$, $\nu _{f}(x)=\|\nabla f(x)\|^{-1}\nabla f(x)$ (and takes an irrelevant arbitrary value if $\nabla f(x)=0$), and $V(f;U)=\int_{U}\|\nabla f(x)\|dx$. If $f=\mathbf{1}_{\{A\}}$ for some $\C^{1}$ compact manifold $A$, $\nu _{f}(x)$ is the outer normal to $A$ for $x\in \partial A$, $\nabla f=0,D^c f=0$, and $D^j f=\mathbf{1}_{\{\partial A\}}\mathcal{H}^{d-1}.$

Denote by $SBV(U)$ the functions $f\in BV(U)$ such that 
$D^c f=0$. Assume  here   that for $m\in \M, g_{m}\in SBV(\mathbb{R}^{d})$, and that $$\int_{\M}\left[
\int_{{\mathbb{R}^{d}}}( | g_{m}(t) | +\|\nabla g_{m}(t)\|) dt
+\int_{J_{g_{m}}}  | g_{m}^{+}(t)-g_{m}^{-}(t) |\mathcal{H}^{d-1}(dt) \right]\mu (dm)<\infty .$$
Let $\N_{0}$ be the class of configurations $\zeta $ such that the corresponding shot noise field $f_{\zeta }$ is of class $SBV(U)$ on every bounded set $U$, finite a.e. on $\mathbb{R}^{d}$, 
its gradient density defined by \eqref{eq:measure-decomp} is a vector-valued shot-noise field, defined a.s. and $\ell^d$-a.e. by 
\begin{align*}
\nabla f_{\zeta }(t  )=\sum_{(x,m)\in \zeta }\nabla g_{m}(t-x),
\end{align*}  and its jump set $J_{f}$ is the union of the translates of the impulse  jump sets: $J_{f}=\cup _{(x,m)\in \zeta }(x+J_{g_{m}}), $ and the jumps of $f$ are  
\begin{align*}
f_{\zeta }^{+}(y )-f_{\zeta }^{-}(y )=\sum_{(x,m)\in \zeta }\mathbf{1}_{\{y\in x+J_{g_{m}}\}}(g_{m}^{+}(y-x)-g_{m}^{-}(y-x)),y\in J_{f}.
\end{align*}
\cite[Theorem 2]{BieDes} and the previous assumption yield that $\eta \in \N_{0}$ a.s..
Let $h$ be a  {\it test function}, i.e. a function $h:\mathbb{R}\to \mathbb{R}$ of class $\mathcal{C}^{1}$  with compact support. Let $H$ be a primitive function of $h$. Bierm\'e and Desolneux \cite[Theorem 1]{BieDes} give for  $W\subset \mathbb{Z} ^{d},\zeta \in \N,$
\begin{align*}
F_{ W}^{h,Per}(\zeta ):=\int_{\mathbb{R}}h(u)\Per(\{f_{\zeta }\geqslant u\};\tilde W)du=F_{  W}^{h,cont}(\zeta )+F_{  W}^{h,jump}(\zeta ),
\end{align*}where 
\begin{align*}
F_{ W}^{h,cont}(\zeta )&=\int_{\tilde W}h(f_{\zeta }(x ))\|\nabla f_{\zeta }(x )\|dx,\\
F_{  W}^{h,jump}(\zeta )&=\int_{J_{f}\cap \tilde W}(H(f_{\zeta }^{+}(x ))-H(f_{\zeta }^{-}(x )))\mathcal{H}^{d-1}(dx).
\end{align*}
 Their expectations under $\eta $ are  computed in \cite[Section 3]{BieDes} :
\begin{align*}
\E [F_{W}^{h,cont}(\eta )]&=  | W |    \E\left[
 h(f_{\eta }(0))\|\nabla f_{\eta }(0)\|
\right]\\
\E [F_{ W}^{h,jump}(\eta )]&=   | W |  \int_{\M}\int_{J_{g_{m}}}\left(
\int_{g_{m}^{-}(y)}^{g_{m}^{+}(y)}\E [h(s+f_{\eta }(0))]ds 
\right)\mathcal{H}^{d-1}(dy)\mu (dm).
\end{align*}
Let us now give their second order behaviour.    
  It is difficult to give sharp necessary conditions for non-degeneracy of the variance if the function $h$ changes signs, so we treat the case $h\geqslant 0  $, but it is can clearly be extended.

\begin{theorem}
\label{thm:sn-per}
Let $\W=\{W_{n};n\geqslant 1\}$ satisfying \eqref{eq:ball-like}. Assume that \eqref{eq:ass-nontrivial-volume} holds and that $\P (F_{W}^{h,Per}(\eta )\neq F_{W}^{h,Per}(\emptyset ))>0$ for some $W\subset \mathbb{Z} ^{d}$. 
Assume that for some 
$\alpha >5d/2,c >0$,  
\begin{align} 
\label{eq:assumption-SN-Per-gradient-gm}
( \E | g_{M}(x) | ^{4})^{1/4} \leqslant c(1+\|x\|)^{-d-\alpha },\\
\label{eq:assumption-SN-Per-gradient}( \E\|\nabla g_{M}(x)\|^{4})^{1/4} \leqslant c(1+\|x\|)^{-d-\alpha },\\
\label{eq:assumption-jumps}
\left(
\E\left[ \hspace{-0.7cm}
\int\limits_{  \hspace{0.8cm}J_{g_{M}}\cap( x+[0,1)^{d}) }\hspace{-0.7cm}( 1\vee | g_{M}^{+}( t)-g_{M}^{-}(  t) |)\mathcal{H}^{d-1}(dt) 
\right]^{4} 
\right)^{1/4}\hspace{-0.5cm}\leqslant c(1+\|x\|)^{-d-\alpha }.
\end{align}
 Then the conclusions of Theorems \ref{thm:ultimate},\ref{thm:ultimate-2},\ref{thm-variance},\ref{thm:BE} hold for  $F_{0}:=F_{\{0\}}^{h,\Per}$. In particular, $F_{W}^{h,\Per}$ has a variance proportional to $ | W | $ and follows a CLT.
 \end{theorem}
 
\begin{example}
 Assume $\M=\mathbb{R}$ is endowed with a probability measure $\mu $ with finite $4$-th moment.  Let  $f$ be a function of the form 
\begin{align*}
f_{\zeta }(x )=\sum_{(y,m)\in \zeta }mg(\|x-y\|)
\end{align*} with $g\in SBV(\mathbb{R})$. Conditions \eqref{eq:assumption-SN-Per-gradient-gm} and \eqref{eq:assumption-SN-Per-gradient}  hold  if $ | g(r) | \leqslant C(1+r)^{-d-\alpha } $  and $ |  g'(r) | \leqslant C(1+r)^{-d-\alpha },r>0$. Then \eqref{eq:assumption-jumps} holds if $J_{g}$ is  countable and for some $C>0,$ for every $r>0$
\begin{align*}
\sum_{t\in J_{g}\cap [r,r+1)}(1\vee | g^{+}(t)-g^{-}(t) | )\leqslant C(1+r)^{-d-\alpha }.
\end{align*}  
\end{example}

\begin{proof}First, \eqref{eq:assumption-SN-Per-gradient-gm}-\eqref{eq:assumption-SN-Per-gradient} imply that the shot noise process and its gradient measure are a.s. well defined. The functionals $F_{W}^{h,cont},F_{W}^{h,jump}$ are under the form \eqref{FW-general-form-Fk}-\eqref{FW-general-form}, with $F_0$ defined respectively by, for $\zeta \in \N,$ 
\begin{align*}
F_{0}^{h,cont}(\zeta )=&\int_{\tilde Q_1}h(f_{\zeta }(t ))\|\nabla f_{\zeta }(t)\|dt\\
F_{0}^{h,jump}(\zeta )=&\int_{J_{f_{\zeta } }\cap \tilde Q_1}(H(f_{\zeta }^{+}(t ))-H(f_{\zeta }^{-}(t) ))\mathcal{H}^{d-1}(dt),
\end{align*}
where $H$ is a primitive function of $h$.
  
 Let $\x_{i}=(x_{i},m_{i})\in \mathbb{R}^{d},i=1,\dots ,6.$
  Let $r>0,  \zeta \subset \{\x_{1},\x_{2}\}  $, and let $\eta _{j}=\eta '\cap A_{j},j=1,2,$ for some $A_{1}\subset A_{2}\subset \mathbb{R}^{d}$ that coincide on $B_{r}$. By the triangular inequality,
  \begin{align*}
\Big|
&F_{0}^{h,cont}(\eta _{1} )-F_{0}^{h,cont}(\eta _{2})
\Big|\leqslant \int_{\tilde Q_1}\|h'\|_{\infty } | f_{\eta _{1} }(t )-f_{\eta _{2}}(t) | \|\nabla f_{\eta_{1} }(t )\|dt\\
&\hspace{5cm}+\int_{\tilde Q_1}\|h\|_{\infty } \|\nabla f_{\eta _{1}}(t )-\nabla f_{\eta _{2}}(t)\|dt\\
\leqslant &\sum_{(x,m)\in \eta  '\setminus B_{r}}\int_{\tilde Q_1}\left[
\|h'\|_{\infty }\|\nabla f_{\eta _{1} }(t)\| | g_{m}(x-t  ) |+\|h\|_{\infty }\|\nabla g_{m}(x-t)\|
\right] dt.
\end{align*}
 Define for $\zeta_{0} \in \N$, $\x=(x,m)\in \overline{\mathbb{R}^{d}}$,
\begin{align*}
\psi (\x,\zeta_{0} )=\int_{\tilde Q_1}\left[\|h'\|_{\infty } \|\nabla f_{ (\zeta_{0}\cup \zeta)\cap A_{1}   }(t) \|| g_{m}(x-t) | +\|h\|_{\infty }\|\nabla g_{m}(x-t)\|\right]dt.
\end{align*} For 
$\zeta '\subset \{\x_{i},3\leqslant i\leqslant 6\}$,  Jensen's inequality yields for $\x=(x,m)\in \overline{\mathbb{R}^{d}}$
\begin{align*}
\E \psi (\x ,\eta \cup   \zeta' )^{4}\leqslant C\int_{\tilde Q_1}\E \left[
 | g_{m}(x -t) | ^{4}\E  \|\nabla f_{ \eta _{1}\cup \zeta ' }(t)\|^{4}+\E \|\nabla g_{m}(x -t)\|^{4}
\right]dt.
\end{align*} 
An easy application of  Lemma \ref{lm:mecke} with $\psi '(x,m)=\|\nabla g_{m}(x-t)\|,r=0$ yields that $\E\| \nabla f_{\eta _{1}\cup \zeta ' }(t)\|^{4} \leqslant c<\infty  $ where  $c$ does not depend on $t\in \mathbb{R}^{d},A_{1}$ or the $\x_{i}$.  
 Therefore, Assumptions \eqref{eq:assumption-SN-Per-gradient-gm} and \eqref{eq:assumption-SN-Per-gradient} yield for $\x=(x,m)\in \overline{\mathbb{R}^{d}}$
\begin{align*}
\E[ \psi (\x ,\eta\cup   \zeta' )^{4}]\leqslant C (1+\|x\|)^{-4(\alpha+d) },
\end{align*}and  Lemma \ref{lm:mecke} with \eqref{eq:sum-F0hcont}  yields that
\begin{align}
\label{eq:sum-F0hcont}
\left(
\E\left[
\Big|
 F_{0}^{h,cont}(\eta  \cap A_{1} )-F_{0}^{h,cont}(\eta \cap A_{2})
\Big | ^{4}
\right]
\right)^{1/4}\leqslant C(1+r)^{-\alpha },
\end{align}where $C$ does not depend on the $A_{i}$. Hence,  \eqref{eq:ultimate-assumption-bis} is satisfied by $F_{0}^{h,cont}$ (hypothetical points of $\zeta \setminus B_{r}$ have to be treated separately).  
 
 Let us now prove that it  is satisfied by the jump functional $F_{0}^{h, jump}$. Since it has to hold only for $\ell^d$-a.e. $x_{1},x_{2}$, and the $J_{g_{1}},J_{g_{2}}$ have finite $\mathcal{H}^{d-1}$ measure, we assume that $J_{g_{m_{1}}}  -x_{1} $ and $J_{g_{m_{2}}}  -x_{2} $ have a $\mathcal{H}^{d-1}-$negligible intersection.  They  also a.s. have a $\mathcal{H}^{d-1}$-negligible intersection with each $J_{g_{m}} -x $, $(x,m)\in \eta $.  Call $f_{1}= f _{\eta _{1}},f_{2}=f_{\eta _{2}}$,  \begin{align}
\notag&\left|
F^{h,jump}_{0}(\eta _{1})-F^{h,jump}_{0}(\eta  _{2})
\right|\\
\notag=&\left|
\sum_{(x,m)\in \eta _{ 1}}\int_{J_{g_{m}}\cap \tilde Q_1}\left[
(H(f_{1}^{+}(t))-H(f_{1}^{-}(t)))-(H(f_{2}^{+}(t))-H(f_{2}^{-})(t))
\right]\mathcal{H}^{d-1}(dt)\right.\\
\notag&-
\left.\sum_{(x,m)\in \eta _{2}\setminus \eta _{1}}\int_{J_{g_{m}}\cap \tilde Q_1}\left[
H(f_{2}^{+}(t))-H(f_{2}^{-}(t))
\right]\mathcal{H}^{d-1}(dt)
\right|\\
\notag\leqslant & \int_{ J_{f_{1}}\cap \tilde Q_1}\|h\|(\left|
f_{2}^{+}(t)-f_{1}^{+}(t)
\right|+\left|
f_{2}^{-}(t)-f_{1}^{-}(t)
\right|)\mathcal{H}^{d-1}(dt)\\
\notag&+\sum_{(x,m)\in \eta' \setminus B _{r}}\int_{\tilde Q_1\cap J_{g_{m}}}\|h\|\left|
g_{m}^{+}(x-t)-g_{m}^{-}(x-t)
\right|\mathcal{H}^{d-1}(dt)\\
\label{eq:sum-F0hjump}\leqslant &\sum_{(x,m)\in \eta' \setminus B _{r}}\|h\|\left(2\underbrace{\int_{J_{f_{\eta '}}\cap \tilde Q_1} | g_{m}(x-t) | \mathcal{H}^{d-1}(dt)}_{=:\psi _{1}((x,m),\eta  )}+
\underbrace{\int_{\tilde Q_1\cap J_{g_{m}}} | g_{m}^{+}(x-t)-g_{m}^{-}(x-t) | \mathcal{H}^{d-1}(dt)}_{=:\psi _{2}(x,m)}
\right).
\end{align}  
We have $\E[ \psi _{2}( x ,M_{0} )^{4}]\leqslant C(1+ \|x\|)^{-4(\alpha+d) }$ by \eqref{eq:assumption-jumps}, and Jensen's inequality yields for $\zeta '\subset \{\x_{3},\dots ,\x_{6}\}$, $f_{3} =f_{\eta '\cup \zeta '  }$, after expanding the $4$-th power of the integral as a quadruple integral,
\begin{align*}
\E \psi _{1}((x ,M_{0}),\eta \cup \zeta ')^{4}=& \E\left(
 \E\left[\left(
\int_{J_{f_{3}}\cap \tilde Q_1} | g_{M_{0}}(x -t) |  \mathcal{H}^{d-1}(dt)
\right)^{4}
\Bigg |\sigma  (\eta ,\zeta ,\zeta ')\right]
\right)\\
\leqslant &\E\left(
 \left(
\int_{J_{f_{3}}\cap \tilde Q_1}(\E g_{M_{0}}( x -t)^{4})^{1/4}\mathcal{H}^{d-1}(dt)
\right)^{4} \Bigg | \sigma (\eta ,\zeta ,\zeta ')
\right)\\
\leqslant &C(1+\|x \|)^{-4(d+\alpha) }\E [\mathcal{H}^{d-1}(J_{f_{3}}\cap \tilde Q_1 ) ^{4}]
\end{align*}
 by Assumption \eqref{eq:assumption-SN-Per-gradient-gm}. Then \eqref{eq:assumption-jumps}   yields $\E[  \mathcal{H}^{d-1}( J_{f_{3 }}\cap \tilde Q_1 ) ^{4}]<\infty $  with an application of Lemma \ref{lm:mecke}, whence  Lemma \ref{lm:mecke} again yields that $F_{0}^{h,jump}$ also satisfies \eqref{eq:ultimate-assumption-bis} (here again the points of $\zeta\cup \zeta '$ have to be considered separately). Hence $F_{0} =F_{0}^{h,cont}+F_{0}^{h,Per}$ satisfies \eqref{eq:ultimate-assumption-bis}.

 It remains to prove Assumption \ref{ass:nondeg}. Assume wlog $F_{0}(\emptyset )=0$. Since a set with positive volume has positive perimeter, Lemma \ref{lm:nontrivial-levelset} and Assumption \eqref{eq:ass-nontrivial-volume} yield $ \rho>1 ,c>0,p>0$ such that for $\beta >\rho $, $\P ( | F_{Q_{\beta }}(\eta ^{\rho })  | \geqslant c)\geqslant p$. Then for $\delta >\gamma >\beta    ,$ 
\begin{align*}
U:=&\left|
F_{Q_{\delta }}(\eta _{\gamma }\cup \eta ^{\rho })-F_{Q_{\delta }}(\eta _{\gamma })-F_{Q_{\beta }}(\eta ^{\rho }) 
\right|\\
\leqslant & \left|
F_{Q_{\beta }}(\eta _{\gamma }\cup \eta ^{\rho })-F_{Q_{\beta }}(\eta ^{\rho })
\right|+ \left|F_{Q_{\delta }\setminus Q_{\beta }}(\eta ^{\rho }\cup \eta _{\gamma })-
F_{Q_{\delta } \setminus Q_{\beta  }}(\eta  _{\gamma })
\right|+ \left|
F_{Q_{\beta }}(\eta _{\gamma }) 
\right|\\
\E U\leqslant&\kappa  \beta ^{d}(\gamma -\beta  )^{-\alpha }+\sum_{m=\beta }^{\delta }\kappa m^{d-1}(m-\rho )^{-\alpha }+\kappa \beta ^{d}(\gamma -\beta )^{-\alpha }\leqslant \kappa \beta ^{d}(\gamma -\beta  )^{d -\alpha }+C_{\rho } (\beta-\rho ) ^{d-\alpha },
\end{align*}
the last estimates are obtained by choosing adequately $A_{1},A _{2}$ in   \eqref{eq:sum-F0hcont},\eqref{eq:sum-F0hjump}. We can arbitrarily increase $\beta $ such that  $C_{\rho } (\beta-\rho ) ^{d-\alpha }<pc/8$, and then for $\gamma $ sufficiently large  $\kappa \beta ^{d}(\gamma -{\beta })^{d-\alpha }<pc/8$ as well, from where 
\begin{align*}
\P ( | F_{Q_{\delta }}(\eta _{\gamma }\cup \eta ^{\rho })-F_{Q_{\delta }}(\eta _{\gamma }) | >c/2)\geqslant &\P ( | F_{Q_{\beta  }}(\eta ^{\rho }) | >c)-\P ( | U | >c/2) \\
\geqslant & p-\E U/(c/2)\geqslant p-p/2=p/2>0.
\end{align*}
That proves Assumption \ref{ass:nondeg} and concludes the proof.
 \end{proof}

\subsection{Fixed level perimeter and Euler characteristic}
\label{sec:sn-general}

\renewcommand{\p}{\mathcal{P}}
\renewcommand{\r}{\mathcal{R}}
Let $\mathcal{B}$ be a measurable subset   of $\mathcal{M}_d $, and let the marks space be  $\M=(\mathbb{R}\setminus \{0\} )\times \mathcal{B}$, endowed with the product $\sigma $-algebra and some probability measure $\mu $.  This section is restricted to shot-noise fields of the form
\begin{align}
\label{eq:sn-indic}
f_{\zeta  }(x )=\sum_{(y,(L,A))\in \zeta  }L\mathbf{1}_{\{x-y\in A\}}, \zeta  \subset   {\mathbb{R}^{d}}\times \M,x\in \mathbb{R}^{d}.
\end{align}
Such fields are used  in image analysis  \cite{BieDesEuler,BieDes}, or in mathematical morphology \cite{Lan02}, sometimes with $L =const.$, and their marginals might not  have a density. The article \cite{BdBDE} uses the asymptotic normality result below for the Euler characteristic when $\mathcal{B}$ is the class of closed discs in $\mathbb{R}^{2}$ (Example \ref{ex:discs}).

The current framework allows to give general  results for a fixed level $u\in \mathbb{R}$, for a large class of additive functionals, including the perimeter or the total curvature, related to the Euler characteristic. For the latter, the main difficulty is to properly define it on a typical excursion of the shot noise field, as it is obtained by locally adding and removing sets from $\mathcal{B}.$   The general result only involves the marginal distribution $\mu _{\mathcal{B}}(\cdot ):=\mu (\mathbb{R}\times \cdot )$.  
 
We call $\mathcal{B}'$ the class of excursion sets generated by shot noise fields of the form \eqref{eq:sn-indic}  where all but finitely many points of $\zeta $ in general position have been removed. 
Formally,
given a  measurable subclass $\mathcal{B}'\subset \mathcal{M}_d $, a function $V :\mathcal{B}'\to \mathbb{R}$ such that $V(A)$ only depends on $A\cap \tilde Q_1$, and a function {  $| V|  :\mathcal{B}\to (0,\infty )$}, say that $(\mathcal{B},\mathcal{B}',V,| V|  )$ is {\it admissible} if  for $ A_{1},\dots ,A_{q}\in  {\mathcal{B}}$, for a.a. $y_{1},\dots ,y_{q}\in \mathbb{R}^{d}$, any set $A$ obtained by sequentially removing, adding or intersecting  the $A_{i}+y_{i}, i=1,\dots ,q,$ belongs to $ {\mathcal{B}'}$, and  $
 |
V (A ) 
 |  \leqslant  \sum_{i= 1}^{q} | V|  (A_{i } ) $.  
 We consider below the functionals, for $W\subset \mathbb{Z} ^{d},\zeta \in \N$,
\begin{align*}
F_{W}(\zeta )=\sum_{k\in W}V(  \{f_{\zeta\cap \tilde W }\geqslant u\}-k ),\;\;\;F_{W}'(\zeta )=\sum_{k\in W}V(\{f_{\zeta  }\geqslant u\}-k).
\end{align*}

\begin{theorem}
\label{thm:fixed-u-EC}
Let $u\in \mathbb{R}$,  $(\mathcal{B},\mathcal{B}',V,| V|  )$ be an admissible quadruple, let $f$ be of the form \eqref{eq:sn-indic}, and let $\W=\{W_{n};n\geqslant1\}$ be a sequence of subsets of $\mathbb{Z} ^{d}$ satisfying \eqref{eq:ball-like}.    Assume that for some $\rho ,p,c>0$, $\P ( | F_{Q_{\beta }}(\eta ^{\rho }) |\geqslant c)\geqslant p$ for $\beta >\rho $,  that  $\int_{\mathcal{B}}| V|  (A)^{8}\mu_{\mathcal{B}} ( dA)<\infty ,$  and that for some $\lambda  >28d, C>0,$ 
\begin{align}
\label{ass:abstract-SN}
\mu _{\mathcal{B}}(\{A\in \mathcal{B}:(x+A)\cap \tilde Q_1\neq \emptyset \}) \leqslant  C(1+\|x\|)^{-\lambda },x\in \mathbb{R}^{d}.
\end{align} 
 Then the
conclusions of Theorems   \ref{thm:ultimate} and \ref{thm:ultimate-2} hold: $F_{W}$ and $F_{W}'$ have  variance of volume order and undergo a CLT. \end{theorem}

Remark that  nothing prevents the typical grain of $ \mathcal{B}$ to be unbounded with positive $\mu_{\mathcal{B}} $-probability.

\begin{proof}  In this proof, $\N_{0}$ is chosen to  be the class of $\zeta $ such that for any bounded set $D$, $\zeta  [D]:=\{(y,(L,A))\in \zeta :(y+A)\cap D\neq \emptyset \}$ is finite. Assumption \eqref{ass:abstract-SN} implies that $\eta \in \N_{0}$ a.s. Let  the notation of \eqref{eq:ultimate-assumption-bis} prevail.    Let $r\geqslant 0 .$ Introduce the independent variables
\begin{align*}
S_{r}^{-}=\sum_{(y,(L,A))\in (\eta ' \cap  B_{r})[ {\tilde Q_1}]}\hspace{-1cm} | V | (A),\;\;S_{r}^{+}=\sum_{(y,(L,A))\in (\eta ' \setminus B_{r})[ {\tilde Q_1}]} \hspace{-1cm}| V | (A).\end{align*}We  have a.s.
\begin{align}
\notag \big |
F_{0}(\eta '\cap B )-F_{0}(\eta '\cap B\cap B_{r})
\big |  =&\left|
V(\{f_{\eta '\cap B\cap B_{r} } \geqslant u\} )-V(\{f_{\eta '\cap B }\geqslant u\} )
\right| \\
\label{eq:fixed-level}\leqslant  & 
 \mathbf{1}_{\{   S_{r}^{+}\neq 0  \}}
2(S_{r}^{-}+S_{r}^{+})\leqslant 2\mathbf{1}_{\{S_{r}^{+}\neq 0\}}S_{r}^{-}+2S_{r}^{+}.
\end{align}
 Define $\psi  (y,(L,A))= \mathbf{1}_{\{(y+A)\cap \tilde Q_1\neq \emptyset \}}| V | (A)$. Let $(L_{0},A_{0})$ be a random variable with law $\mu $. We have by Cauchy-Schwarz, for $y\in \mathbb{R}^{d},$
\begin{align*}
\E (\psi (y,(L_{0},A_{0})))^{4}\leqslant \sqrt{\E  | V | (A_{0})^{8}}\sqrt{\P ((y+A_{0})\cap \tilde Q_1\neq \emptyset )}\leqslant C(1+\|y\|)^{-\lambda /2}.
\end{align*} 
Hence Lemma \ref{lm:mecke} yields $\sup_{r>0}\E (S_{r}^{-})^{4}\leqslant \E (S_{\infty }^{-})^{4}<\infty $. The same method yields $(\E (S_{r}^{+})^{4})^{1/4}\leqslant C(1+r)^{-\lambda /8+d }$. The same method again but this time with $\psi (y,(L,A))=\mathbf{1}_{\{(y+A)\cap \tilde Q_1=\emptyset \}}$ yields $\P (S_{r}^{+}\neq 0)\leqslant C(1+r)^{-\lambda  +d }.$ Taking the fourth moment and plugging these estimates back in \eqref{eq:fixed-level} yields that  \eqref{eq:ultimate-assumption-bis} and \eqref{eq:ultimate-assumption} hold. 
 
Let us show that Assumption \ref{ass:nondeg} holds. For $\beta >\rho  $, $\P (  | F_{Q_{\beta }}(\eta ^{\rho }) | \geqslant c)>p$.
 If $\eta _{\gamma }[Q_{\beta }]=\emptyset $,
$
F_{Q_{\beta }}(\eta ^{\rho }\cup \eta _{\gamma })=F_{Q_{\beta  }}(\eta ^{\rho }) $
 and if $\eta ^{\rho }[Q_{\beta }^{c}]=\emptyset $,
 $
F_{Q_{\delta }\setminus Q_{\beta }}(\eta ^{\rho }\cup \eta _{\gamma })=F_{Q_{\delta }\setminus Q_{\beta }}(\eta _{\gamma }).
 $
Hence, with $U_{\delta ,\gamma }:=F_{Q_{\delta }}(\eta ^{\rho })-F_{Q_{\delta }}(\eta ^{\rho }\cup \eta _{\gamma })-F_{Q_{\beta }}(\eta ^{\rho })$,
\begin{align*}
\P ( | F_{Q_{\delta }}(\eta ^{\rho })-F_{Q_{\delta }}(\eta ^{\rho }\cup \eta _{\gamma }) | >c/2)\geqslant& \P ( | F_{Q_{\beta }}(\eta ^{\rho }) | >c)-\P ( | U_{\delta ,\gamma } | >c/2)\\
\geqslant& p-\P (\eta ^{\rho }[Q_{\beta }^{c}]\neq \emptyset )-\P (\eta _{\gamma }[Q_{\beta }]\neq \emptyset ).
\end{align*}
Since at fixed $\rho ,$ $\mathbf{1}_{\{\eta ^{\rho }[Q_{\beta }^{c}]\neq \emptyset \}}\to 0$ a.s. as $\beta \to \infty $, fix $\beta $ such that $\P (\eta ^{\rho }[Q_{\beta }^{c}]\neq \emptyset )<p/4. $ Then for $\gamma $ sufficiently large and any $\delta >\gamma $, $\P (\eta _{\gamma }[Q_{\beta }]\neq \emptyset )<p/4$, hence Assumption \ref{ass:nondeg} is satisfied.
 
  \end{proof}

\begin{example}[Volume]
The simplest example is the class  $\mathcal{B}= \mathcal{M}_d $ of measurable subsets of $\mathbb{R}^{d}$, endowed with Lebesgue measure $V(A )=\ell^d(A\cap  \tilde Q_1).$  We  have $F_{W}(\eta ):=\ell^d(\{f_{\eta \cap \tilde W}\geqslant u\}\cap \tilde W)$ a.s..  This example has been treated in a different framework at Section \ref{sec:SN-volume}.

\end{example}

\begin{example}[Perimeter]
Let $\mathcal{B}$ be the class of $A\in \mathcal{M}_d $ such that $\mathcal{H}^{d-1}(\partial A)<\infty $.   Define $V(A)= \mathcal{H}^{d-1}(\partial A\cap \tilde Q_1),$ we prove below that  $F_{W}(\eta )= \mathcal{H}^{d-1}(\{f_{\eta \cap \tilde W}\geqslant u\}\cap \tilde W)$ a.s.. Assume for the moment condition that $\int_{\mathcal{B}}\mathcal{H}^{d-1}(\partial A)^{8}\mu _{\mathcal{B}}(dA)<\infty .$

\end{example}

\begin{example}[Total curvature] 
\label{ex:discs}
Let $d=2$, $\mathcal{B}$ be the class of non-trivial closed discs of $\mathbb{R}^{2}$. A set $A\subset \mathbb{R}^{2}$ is an {\it elementary set} in the terminology of Bierm\'e \& Desolneux \cite{BieDesEuler} if $\partial A$ can be decomposed as a finite union of $\mathcal{C}^{2}$ open curves $C_{j},j=1,\dots ,p$ with respective constant curvatures $\kappa _{j}>0$, separated by corners $x_{i}\in \partial A,i=1,\dots ,q,$ (with $0\leqslant  q\leqslant p$) with angle  $\alpha (x_{i},A)\in (-\pi ,\pi )$. The \emph{ total curvature} of $A$ within some open set $U$ is defined by
\begin{align*}
TC(A;U):=\sum_{j=1}^{p}\kappa _{j}\mathcal{H}^{1}(C_{j}\cap  U )+\sum_{i=1}^{q}\mathbf{1}_{\{x_{i}\in U\}}\alpha (x_{i},A).
\end{align*}Therefore we define $V(A)=TC(A;\tilde Q_{1})$. Via the Gauss-Bonnet theorem, for $W\subset \mathbb{Z} ^{d},$ $TC(A;\text{\rm{int}}(\tilde W))$
 is strongly related to the Euler characteristic of $A\cap \tilde W$, in the sense that they coincide if $A\subset \text{\rm{int}}(\tilde W)$, and otherwise they only differ by boundary terms, see \cite{BieDesEuler}. We will see that $F_{W}(\eta )=TC(\{f_{\eta \cap \tilde W }\geqslant u\};\text{\rm{int}}(\tilde W))$ a.s.. Assume also that the typical radius has a finite moment of order $ 8d$.\end{example}

\begin{proposition}
\label{prop:volumes-token}
In the three previous examples, assume that \eqref{ass:abstract-SN} holds, and that  $\P ( f_{\eta }(0)\geqslant cu)\notin \{0,1\}$ for some $c>0$. Then the functionals $F_{W},F'_{W}$ satisfy the conclusions of theorems \ref{thm:ultimate},\ref{thm:ultimate-2}, in particular, they have  variance of volume order and undergo a central limit theorem as $ | \partialZ W | / | W | \to 0.$
\end{proposition}

\begin{proof}
All proofs rely on defining an admissible quadruple that satisfies the assumptions of Theorem \ref{thm:fixed-u-EC}, and show that the variance assumption holds. We only treat the case $u\geqslant 0$, the case $u\leqslant 0$ can be treated similarly. Let $\Gamma _{k}=k+\tilde Q_1,k\in \mathbb{Z} ^{d}.$

   (Volume) Defining $\mathcal{B}'=\mathcal{M}_d ,| V|  (A)=\ell^d(A)$ yields an admissible quadruple $(\mathcal{B},\mathcal{B}',V,| V|  )$. In the case $u>0$, the fact that $\mathbf{1}_{\{f_{\eta }(0)>0\}}$ is not trivial yields that \eqref{eq:ass-nontrivial-volume} holds, and hence using Lemma \ref{lm:nontrivial-levelset},  $\P (   F_{Q_{\beta }}(\eta ^{\rho }) \geqslant c)\geqslant \P (   F_{Q_{\rho  }}(\eta ^{\rho }) \geqslant c) =:p>0$ holds for some $\rho , c>0$, and for $\beta >\rho $.  The case $u=0$ can be treated directly and is left to the reader.

 (Perimeter)   Let $\mathcal{B}'$ be the class of $A\in \mathcal{B}$ such that $\mathcal{H}^{d-1}(\partial A\cap \partial \Gamma _{k})=0 $ for $k\in \mathbb{Z} ^{d}$.   For $A\in \mathcal{B}$, for a.a. $y \in \mathbb{R}^{d}$, $\mathcal{H}^{d-1}(\partial (A+y)\cap \partial \Gamma _{k})=0$. 
 Hence for $A_{1},\dots ,A_{q}\in \mathcal{B}$, for a.a. $y_{1},\dots ,y_{q}\in \mathbb{R}^{d}$, any  set $A$ obtained by sequentially adding, intersecting or removing the $A_{i}+y_{i}$ is in $\mathcal{B}'$, using  
$\partial  A \subset \cup _{i=1}^{n}(\partial  A_{i}+y_{i})$. Defining $| V|  (A):=\mathcal{H}^{d-1}(\partial A)
$ yields an admissible quadruple $(\mathcal{B},\mathcal{B}',V,| V|  )$. The justification that   $\var(F_{Q_{\beta }}(\eta ^{\rho }))>0$ holds is the same as for the volume (above), because a set with positive volume has positive boundary measure.

 (Total curvature) Let $\mathcal{B}'$ be the class of sets obtained from finite unions, intersections and removals of discs $A_{1},\dots ,A_{q}$ such that for $i\neq j$, $A_{i}$ and $A_{j}$ are not tangent and $\partial A_{i}\cap \partial A_{j}\cap \partial \Gamma _{k}=\emptyset $ for $k\in \mathbb{Z} ^{d}$.   Every $A\in \mathcal{B}'$ is elementary, and  defining $| V|  \equiv 1$ yields that $(\mathcal{B},\mathcal{B}',V,| V|  )$ is an admissible quadruple. 
Let  $X_{i}=(Y_{i},(L_{i},D_{i})),i\geqslant 1,$ iid marked couples of discs with iid uniform centers $Y_{i}$ in $B(0,1)$. Let $k\in \mathbb{N}$ be such that the event $\Gamma =(\sum_{i=1}^{k}L_{i}\geqslant u,\sum_{i=1}^{k-1}L_{i}<u)$ has positive probability. Conditionally on $\Gamma $, $\{f_{\{X_{1},\dots ,X_{k}\}}\geqslant u\}=\cap _{i=1}^{k} (Y_{i}+D_{i})$. Since the $D_{i}$ have positive radii, the probability that the $Y_{i},i=1,...,k$ are sufficiently close to $0$ such that this set is non-empty is also positive. In this case it is the intersection of discs, hence its total curvature is equal to $1$, and $\P  (F_{Q_{\beta }}(\eta ^{\rho }) \geqslant 1)\geqslant p>0$ is satisfied for some $\rho >0$ and $\beta >\rho $.

\end{proof}

With a similar route, the previous example can likely be generalised to more general classes of sets $\mathcal{B}$ in higher dimensions, such as the polyconvex ring, provided one can estimate properly the curvature or the Euler characteristic on sets from $\mathcal{B}'$.

\section{Proofs}
Recall that $\kappa $ denotes a constant which depends on $d, \alpha ,a_{-},a_{+}$   and whose value might change from line to line. 
The following lemma is useful several times in the paper.   
\begin{lemma} 
\label{lm:mecke}Let $\alpha >d,C_{0}\geqslant 0$, $M_{i},0\leqslant i\leqslant 4$ be
independent marks  with law $\mu $. Let $r>0, \psi :\overline{\mathbb{R}^{d}} \times \N\to \mathbb{R}_{+}$ be a measurable function such that for $\ell^d$-a.e. $x_{i} \in \mathbb{R}^{d},0\leqslant i\leqslant 4,$   and $\zeta \subset \{ (x_{i},M_{i}),i=1,\dots ,4\}$,
$ 
\left(
\E  \psi ( (x_{0},M_{0}),\eta\cup \zeta  )^{4}
\right)^{1/4}\leqslant C_{0} (1+\|x_{0}\|)^{-\alpha -d}
.$  Then 
\begin{align*}
\left(
\E\left|
\sum_{\x\in \eta \setminus B_{r}}\psi (\x;\eta )
\right|^{4}
\right)^{1/4}\leqslant C_{0} \kappa (1+r)^{- \alpha }.
\end{align*} 
\end{lemma}

 {
\allowdisplaybreaks
 \begin{proof}Let $\eta _{r}=\eta \setminus B_{r}$. Let $\x_{i}=(x_{i},M_{i})$. Let $\mathcal{P}_{4} $ be the family of ordered tuples of natural integers which sum is $4$. The multi-variate Mecke formula yields
\begin{align*} 
\E&\left[
\sum_{\x\in \eta_r }\psi (\x;\eta )
\right]^{4}\leqslant \kappa  \sum_{(m_{1},\dots ,m_{q})\in \mathcal{P}_{4}}\E\left[
\sum_{ (\x_{1},\dots ,\x_{q})\in  \eta_r ^{q} }\psi (\x_{1};\eta )^{m_{1}}...\psi (\x_{q};\eta )^{m_{q}}
\right]\\ 
\leqslant &\kappa \sum_{ (m_{1},\dots ,m_{q})\in \mathcal{P}_{4}} \int_{(\overline{B_{r}^{c}})^{q }}\E\left[
\prod_{l=1}^{q} \psi (\x_{l},\eta \cup \{\x_{1},\dots ,\x_{q}\} )^{m_{l}}
\right] d\x_{1}\dots d\x_{q }\\
\leqslant&\kappa \sum_{ (m_{1},\dots ,m_{q})\in \mathcal{P}_{4}} \int_{(\overline{B_{r}^{c}})^{q}}\prod_{l=1}^{q}(\E \psi (\x_{l},\eta  \cup  \{\x_{1},\dots ,\x_{q}\})^{4} )^{m_{l} /4}d\x_{1}\dots d\x_{q}\\
\leqslant  &\kappa \sum_{ (m_{1},\dots ,m_{q})\in \mathcal{P}_{4}} \prod_{l=1}^{q}\kappa \int_{B_{r}^{c}}C_{0}^{m_{l} }(1+\|x_{l}\|)^{-m_l(\alpha+d) }d x_{l}\\
\leqslant &\kappa \sum_{ (m_{1},\dots ,m_{q})\in \mathcal{P}_{4}} C_{0}^{4}\prod_{l=1}^{q}\kappa \int_{a_{-}r}^{\infty }(1+t)^{-m_{l} (\alpha+d) }t^{d-1}dt\\
\leqslant &\kappa C_{0}^{4}\sum_{ (m_{1},\dots ,m_{q})\in \mathcal{P}_{4}} (1+r)^{-4(\alpha+d) +qd}\leqslant \kappa C_{0}^{4}(1+r)^{-4\alpha } .
\end{align*}
\end{proof}
}

\subsection{Proof of Theorem \ref{thm-variance}}
 \label{sec:proof-variance} 
 
 We prove \eqref{eq:cov-bound} under Assumption \eqref{ass:upper-variance-2} (i.e. in case {\bf (i)}).  Remark first  that  \eqref{ass:upper-variance-2} trivially holds also for $B\in \mathcal{B}_{W}^{s}\setminus \mathcal{B}_{W}^{r},s< r.$ Also, if \eqref{ass:upper-variance-2} is satisfied, with $B=\mathbb{R}^{d}$, \eqref{ass:upper-variance} is also satisfied. Assume without loss of generality $F_{0}(\emptyset )=0,$ then \eqref{ass:upper-variance-2} with $r=0$ yields  
\begin{align*}
m_{2}:=&  \sup_{ k\in W}\E |  F_{k}^{W}(\eta ) | ^{2} =\sup_{k\in W} \E | F_{0}((\eta \cap \tilde W)-k)-F_{0}(((\eta \cap \tilde W)-k)\cap B_{0}) |^{2} \leqslant \kappa C_{0}^{2}  <\infty .
\end{align*}

 The following inequality is useful several times in the proof: given some square-integrable random variables $Y_{i},Z_{i},i=1,2$ on $\Omega $, and a $\sigma $-algebra $\mathcal{Z}\subset \mathscr{A}$,   
 \begin{align}
\notag\E&\left|
\cov(Y_{1},Y_{2} | \mathcal{Z})-\cov(Z_{1},Z_{2} | \mathcal{Z})
\right|\\
\notag\leqslant &\E\left(
 \sqrt{ 2\E (Y_{1}^{2} | \mathcal{Z})}\sqrt{2\E ((Z_{2}-Y_{2})^{2} | \mathcal{Z}) }+
  \sqrt{2\E(Z_{2}^{2} | \mathcal{Z})}\sqrt{2\E ((Z_{1}-Y_{1})^{2} | \mathcal{Z}) }
\right) \\
\label{eq:approx-cov}\leqslant &2 \left(
 \sqrt{\E Y_{1}^{2}   }\sqrt{\E (Z_{2}-Y_{2})^{2} }+
  \sqrt{\E  Z_{2}^{2}  }\sqrt{\E (Z_{1}-Y_{1})^{2} }
\right) .
\end{align}

Let $B_{r}(k)=k+B_{r}$ for $k\in \mathbb{Z} ^{d},r\geqslant 0.$
  Let $k,j\in W ,r=\|k-j\|/(3a_{+})$, $\eta',\eta '' $  independent copies of $\eta $, and 
\begin{align*}
  \eta _{k}=(\eta \cap B_{r}(k))\cup (  \eta' \cap B_{r}(k)^{c}),\;\;\;\eta _{j}=(\eta \cap  B_{r} (j))\cup (  \eta'' \cap B_{r}(j)^{c}),
\end{align*} which are  processes distributed as $\eta $, independent since $B_{r}(k)\cap B_{r}(j)=\emptyset $.  Since $\eta \cap B_{r}(k)=\eta _{k}\cap B_{r}(k)$, \eqref{ass:upper-variance-2} yields
\begin{align*}
F_{k}^{W}(\eta )-F_{k}^{W}(\eta _{k})=&F_{k}^{W}(\eta )-F_{k}^{W}(\eta \cap B_r(k))+F_{k}^{W}(\eta _{k}\cap B_r(k))-F_{k}^{W}(\eta _{k})\\
\E\left|
F_{k}^{W}(\eta  )-F_{k}^{W}(\eta_{k}  )
\right|^{2}\leqslant &2\big(
\E\left|
F_{0} ((\eta-k )\cap( \tilde W-k))-F_{0}((\eta-k )\cap ( \tilde W-k)\cap B_{r}  )
\right|^{2}\\
&+\E\left|
F_{0}((\eta_{k}-k )\cap ( \tilde W-k)\cap B_{r}  )-F_{0} ((\eta _{k}-k)\cap ( \tilde W-k) )
\right|^{2}
\big)\\
\leqslant&  \kappa C_{0}^{2}(1+ r )^{-2\alpha },
\end{align*} because $\eta _{k}-k\equlaw \eta -k\equlaw \eta $. A similar  bound holds  for $F_{j}^{W}$. Then,
\eqref{eq:approx-cov} yields 
\begin{align}
\notag
\bigg|
\cov(F^W _{k} (\eta )&,F^{W} _{j}(\eta ) )-\underbrace{\cov(F^W _{k} (\eta _{k}),F^{W} _{j} ( \eta_{j}))}_{=0}
\bigg|\leqslant  \kappa \sqrt{\E F^{W} _{j}(\eta )^{2 }}\sqrt{ \E \left|
F^W _{k}(\eta ) -F^W _{k}( \eta_{k} ) 
\right|^{2}}\\
\notag&\hspace{4cm}+  \kappa \sqrt{\E (F^W _{k}(\eta  _{k}))^{2 }}\sqrt{ \E \left|
F^{W  } _{j}(\eta ) -F^{W} _{j}(\eta _{j} ) 
\right|^{2}}\\ 
\label{eq:cov-Fq}\leqslant & \kappa \sqrt{m_{2}}\sqrt{C_{0}^{2}(1+r)^{-2\alpha }}
\leqslant \kappa C_{0}^{2}(1+\|k-j\|)^{-\alpha }.
\end{align} 
Hence \eqref{eq:cov-bound} is proved in case {\bf (i)}.     If $G_{k}^{W}=F_{k}$  and \eqref{ass:upper-variance} is assumed instead of \eqref{ass:upper-variance-2} (case {\bf(i')}), replacing $W$ by $\mathbb{Z} ^{d}$ in the computation above
yields the same bound  for $\cov(F_{k},F_{j})$. The finiteness of $\sigma _{0}$ follows from $\alpha >d$.

Let us  now assume $ | W | <\infty $ and show \eqref{eq:variance-exact-asymp}. Let $k\in W,r=d(k,\tilde W^{c})/a_{+}$, so that $B_{ r}\cap (\tilde W-k)=B_{r}$. We have if \eqref{ass:upper-variance-2} holds \begin{align*}
  F_{k}^{W}-F_{k} &=  F_{0}((\eta -k)\cap (\tilde W-k) )-F_{0}((\eta -k)\cap (\tilde W-k)\cap B_{r})   +   F_{0}((\eta -k)\cap B_{r})-F_{0}(\eta -k)   \\
 \E| F_{k}^{W}-F_{k} |^{2}&\leqslant \kappa  C_{0}^{2}(1+r)^{-2\alpha }\leqslant \kappa C_{0}^{2}(1+d(k,\tilde W^{c}))^{-2\alpha }.
\end{align*}  We hence have  by \eqref{eq:approx-cov}, for $k,j\in W$, recalling also \eqref{eq:cov-Fq},
\begin{align}
\notag
\left|
\cov(F_{k}^{W},F_{j}^{W})-\cov(F_{k},F_{j})
\right|\leqslant& \kappa C_{0}^{2}(1+\min(d(k,\tilde W^{c}),d(j,\tilde W^{c})))^{-\alpha }\\
\label{eq:diff-covFkFkW}\leqslant &\kappa C_{0}^{2}(1+\max(\|k-j\|,\min(d(k,\tilde W^{c}),d(j,\tilde W^{c}))))^{-\alpha }.
\end{align}Denote by $[x]$ the integer part of $x\in \mathbb{R}.$ Let $ d_{ {W}}\in \mathbb{N}\setminus\{0\}$, $W_{m}=\{k\in W:[d(k,\tilde W^{c})]=m\}\text{\rm{ for }}m\in \mathbb{N},W_{\partial}=\{k\in W:[d(k,\tilde W^{c})]\leqslant  d_{ {W}}\}=\cup _{m=0}^{d_{W}}W_{m},W_{int}=W\setminus W_{\partial}$.  
 We have, using \eqref{eq:cov-bound} and \eqref{eq:diff-covFkFkW}, 
\begin{align*}
 | \var(F_{W})&-\sigma _{0}^{2} | W |  | =\left|
\sum_{k\in W,j\in W}\cov(F_{k}^{W},F_{j}^{W})-\sum_{k\in W,j\in \mathbb{Z} ^{d}}\cov(F_{k} ,F_{j} )
\right|\\
\leqslant&  \sum_{k\in W,j\notin W}  | \cov(F_{k} ,F_{j} ) | +2\hspace{-0.5cm}\sum_{k,j\in W:d(k,\tilde W^{c})\leqslant d(j,\tilde W^{c})} \hspace{-0.5cm}\left | 
 \cov(F_{k}^{W},F_{j}^{W})- \cov(F_{k},F_{j})
\right | \\
 \leqslant &  \sum_{m=0}^{\infty }\sum_{k\in W_{m}}\left[
\sum_{j\in W^{c}}\kappa C_{0}^{2}(1+\|k-j\|)^{-\alpha }+2\sum_{j\in W }\kappa C_{0}^{2}(1+\max(\|k-j\|,m))^{- \alpha } 
\right]\\
 \leqslant & \kappa C_{0}^2 \sum_{m=0}^{\infty }\sum_{k\in W_{m}}\left(3\sum_{j\in B(k,m)^{c}} (1+\|k-j\|)^{-\alpha }  +2
\sum_{j\in B(k,m)} (1+m)^{- \alpha  } \right)\\ 
 \leqslant & \kappa C_{0}^2 \sum_{m=0}^{\infty }\sum_{k\in W_{m}}\left(3\kappa  (1+m)^{-\alpha +d}  +2\kappa 
 m^{d}(1+m)^{-\alpha } \right)\\ 
 \leqslant &\kappa C_{0}^2 \left(
 | W_{\partial} | +d_{W}^{-\alpha +d} | W_{int} | 
\right)
\end{align*}hence 
$ \left|
\frac{  \var(F_{W})}{ | W | }-\sigma _{0}^{2}
\right|\leqslant \kappa C_{0}^{2}\left(
\frac{d_{W}^{d} | \partialZ W | }{ | W | }+d_{W}^{-\alpha +d}
\right).$
 Equation \eqref{eq:variance-exact-asymp} follows  {by taking $d_{{W}}=[( | W | / | \partialZ W | )^{\frac{1}{\alpha }}]$. The same computation where $F_{k}^{W}$ is replaced by $F_{k}$ (hence with no second term on the second line), treats the case {\bf (i')}, without requiring \eqref{ass:upper-variance-2}.} 

Let us now prove that under the current assumptions, Assumption \ref{ass:nondeg} implies $\sigma _{0}>0$. Recall the notation
 $\eta _{a}=\eta \cap \tilde Q_{a}^{c},\eta _{a}^{b}=\eta_{a} \cap \tilde Q_{b},a,b>0$. Let $\delta   >0$, and decompose $W$ in the finite disjoint union of subparts with sidelength $\delta $: $W=\cup _{k\in \mathbb{Z} ^{d}}W^{(k)}$ where $W^{(k)}=W\cap (\delta k +Q_{\delta } )$. Decompose accordingly $F_{W} =\sum_{k\in \mathbb{Z} ^{d}}F^{(k)}$
where $F^{(k)}=\sum_{j\in W^{(k)}}F_{j}$. Let $\gamma<\delta ,$ and condition by the points of $\eta $ $\gamma $-close to the boundary of a $W^{(k)}$:
 $
\eta_{\gamma } ^*=\eta \cap \tilde  Q_{\gamma }^*$ where $ \tilde  Q_{\gamma }^*=\mathbb{R}^{d}\setminus (\cup _{k\in \mathbb{Z} ^{d}} (\delta k +\tilde Q_{\gamma }) ).
 $ Denote by $\E _{\eta _{\gamma }^{*}},\var_{\eta _{\gamma }^{*}}$ and $\cov_{\eta _{\gamma }^{*}}$ the conditional expectation, variance, and covariance with respect to $\eta _{\gamma }^{*}.$
We have 
\begin{align}
\label{eq:var-lb-1}
\text{\rm{\var}}(F_{W} )\geqslant \E[ \Var_{\eta_{\gamma } ^*}(F_{W})] \geqslant \sum_{k\in \mathbb{Z} ^{d}}\E [\var_{\eta_{\gamma } ^*}(F^{(k)})]-\sum_{k\neq j}\E |  \cov_{\eta_{\gamma } ^*}(F^{(k)},F^{(j)}) | .
\end{align}
  We claim (and prove later) that for $k\in \mathbb{Z} ^{d}$ \begin{align}
\label{eq:cov-lb}
\E \sum_{j\neq k} | \cov_{\eta_{\gamma } ^*}(F ^{(k)},F^{(j)}) | \leqslant C ' \delta ^{2d}(\delta -\gamma)   ^{  -\alpha }.
\end{align}
For the first term of \eqref{eq:var-lb-1}, among the $k\in \mathbb{Z} ^{d}$ such that $W^{(k)}\neq \emptyset $, call $W^{\delta ,int}$ those such that $W^{(k)}-k\delta =Q_{\delta }$, and $W^{\delta ,\partial}$ the others.  We have, using also \eqref{eq:approx-cov},
\begin{align}
\notag\sum_{k\in \mathbb{Z} ^{d}}\E[ \var_{\eta_{\gamma } ^*}(F^{(k)})]\geqslant& \sum_{k\in W^{\delta ,int}}\E [\var_{\eta_{\gamma } ^*}(F^{(k)})]-2\sum_{k\in   W^{\delta ,\partial}}\E [(F^{(k)})^{2}]\\
\label{eq:var-lb-2}
\geqslant &| W^{\delta ,int}|\E[ \var_{\eta_{\gamma } ^*}(F_{Q_{\delta  }})]-2| W^{\delta ,\partial}| \delta  ^{d}m_{2}
\end{align}because by stationarity, for $k\in W^{\delta ,int}$,  $\E[ \var_{\eta_{\gamma } ^*}(F^{(k)})]=\E [\var_{\eta_{\gamma } ^*}(F_{Q_{\delta   }})]$.
 Recall that any real random variable $U$ satisfies $\var(U)=\inf_{z\in \mathbb{R}}\E (U-z)^{2}$. Since $\tilde  Q_{\gamma }^*\subset Q_{\gamma }^{c}$, $\eta_{\gamma } ^*\in \sigma (\eta _{\gamma }),$ hence for $\rho \in (0,\gamma )$, \begin{align*}
\E[ \var_{\eta_{\gamma } ^*}(F _{Q_{\delta }})]\geqslant &\E[ \var_{\eta _{\gamma }}(F_{Q_{\delta   }})]=\E \left[
\inf_{z\in \mathbb{R}}\E _{\eta _{\gamma }}(F_{Q_{\delta   }}-z)^{2}\right]\\
\geqslant & \E 
\left[
\inf_{z\in \mathbb{R}}\E_{\eta _{\gamma }} [ 1_{\eta _{\rho }^{\gamma }=\emptyset }(F_{Q_{\delta  }}-z)^{2} ]
\right] =\P (\eta _{\rho }^{\gamma }=\emptyset )\E \left[
\inf_{z\in \mathbb{R}}\E_{\eta _{\gamma }} [( F_{Q_{\delta   }}(\eta ^{\rho }\cup \eta _{\gamma })-z)^{2}]
\right]\\
=&\P (\eta _{\rho }^{\gamma }=\emptyset )\E\left[
 \var_{\eta _{\gamma }}[F_{Q_{\delta   }}(\eta ^{\rho }\cup \eta_{\gamma })]
\right]
\end{align*}
where the second equality is true because $\eta _{\rho }^{\gamma } ,\eta ^{\rho },\eta _{\gamma  }$ are independent and $1_{\{\eta _{\rho }^{\gamma }=\emptyset\} }F_{Q_{\delta }} =\mathbf{1}_{\{\eta _{\rho }^{\gamma }=\emptyset \}}F_{Q_{\delta }}(\eta ^{\rho }\cup \eta _{\gamma  })$. Up to increasing $\delta $, let $0<\rho <\gamma $ be like in Assumption \ref{ass:nondeg}, which yields  $v_{\gamma }>0$ such that for arbitrarily large $\delta >\gamma $, $\E[\var_{\eta _{\gamma }}(F_{Q_{\delta }}(\eta ^{\rho }\cup \eta _{\gamma }))]\geqslant v_{\gamma }$. By \eqref{eq:cov-lb}, \eqref{eq:var-lb-1} and \eqref{eq:var-lb-2} for $\delta >\gamma   $ sufficiently  large
\begin{align*} \var(F_{W})\geqslant &|W^{\delta ,int}|\P (\eta _{\rho }^{\gamma }=\emptyset )v_{\gamma }-2|W^{\delta ,\partial}|\delta ^{d}m_{2}-(|W^{\delta ,int}|+| W^{\delta ,\partial}|)C'\delta ^{2d}(\delta -\gamma )^{-\alpha }\\
\geqslant &|W^{\delta ,int}|(\P (\eta _{\rho }^{\gamma }=\emptyset )v_{\gamma }-C'\delta ^{2d}(\delta -\gamma )^{-\alpha })-  | W^{\delta ,\partial}|(2\delta ^{d}m_{2}+C'\delta ^{2d}(\delta -\gamma )^{-\alpha }).
\end{align*}
Since $\alpha >2d$, given any $\gamma $, one can choose $\delta =:\delta_{\gamma } $ such that $C'\delta ^{2d}(\delta -\gamma )^{-\alpha }<\varepsilon _{\gamma }:=\P (\eta _{\rho }^{\gamma  }=\emptyset )v_{\gamma }/2$. Hence $\var(F_{W})\geqslant  | W^{\delta ,int}| \varepsilon _{\gamma }- |  W^{\delta ,\partial} | (2\delta ^{d}m_{2}+\varepsilon _{\gamma }).$  To conclude, let a sequence $\{W_{n};n\geqslant 1\}$ be such that $\lim_{n} | \partial _{\mathbb{Z} ^{d}}W_{n} | / | W_{n}  | = 0$. Since $ | \partialZ W_{n} | / | W_{n} | \geqslant  | W_{n}^{\delta ,\partial } | /(\delta ^{d} (| W_{n} ^{\delta ,int}|+ | W_{n}^{\delta ,\partial } |)  )$, \eqref{eq:variance-exact-asymp} yields
\begin{align*}
\sigma _{0}=\liminf_{n} | W_{n} | ^{-1}\Var(F_{W_{n}})\geqslant  \liminf_{n} (\delta ^{d}| W_{n}^{\delta ,int} |) ^{-1}\var(F_{W_{n}})>0.
\end{align*} 

 Let us finally prove \eqref{eq:cov-lb}.   Let $k\neq j\in \mathbb{Z} ^{d}, l\in W^{(k)},m\in W^{(j)},r=\|j-k\| (\delta -\gamma )/(2a_{+})$. Let $\eta ',\eta ''$  independent copies of $\eta $, and define
\begin{align*}
\eta _{l}=&(\eta \cap B_r(l) )\cup ( \eta' \cap B_{r }(l)^{c}),\;\;\;\eta _{m}= (\eta \cap  B_{r}(m))\cup (\eta'' \cap B_{r}(m)^{c}).
\end{align*}
Since $ B_r(l) \cap B_{r }(m)\subset \tilde  Q_{\gamma }^*$, $\eta _{l}$ and $\eta _{m}$ are independent conditionally to $\eta_{\gamma } ^*,$ and we have by \eqref{eq:approx-cov}, with a computation similar to \eqref{eq:cov-Fq},
 $
\E  | \cov_{\eta_{\gamma } ^*}(F_{l},F_{m})- {\cov_{\eta_{\gamma } ^*}(F_{l}(\eta _{l}),F_{m}(\eta _{m}))}  | \leqslant \kappa C_{0}^{2}(1+r)^{-\alpha }. $
It follows that  
\begin{align*}
\E \left|\cov_{\eta_{\gamma } ^*}(F ^{(k)},F ^{(l)})
\right|\leqslant \E \sum_{l\in W^{(k) },m\in W^{(j) }} | \cov_{\eta_{\gamma } ^*}(F_{l},F_{m}) | \leqslant \kappa C_{0}^{2}\delta ^{2d}(1+r )^{- \alpha }
\end{align*}
and, for some $C'$ not depending on $W,$ for $k\in W,$
\begin{align*}
\E \sum_{j\in \mathbb{Z} ^{d}\setminus \{k\}}\left|
\cov_{\eta _{\gamma }^{*}}(F^{(k)},F^{(j)})
\right|\leqslant \kappa C_{0}^{2} \delta ^{2d} \sum_{p=1}^{\infty }p^{d-1  }( \|p\| (\delta -\gamma ))^{-\alpha } \leqslant C '\delta  ^{2d}(\delta -\gamma )^{-\alpha }.
\end{align*} This concludes the proof of \eqref{eq:cov-lb} and hence of $\sigma _{0}>0.$
 
It remains to prove  \eqref{eq:moments-upper-bound}. Assume that \eqref{ass:upper-variance}  holds with $\alpha >2d$. The proof when instead  \eqref{ass:upper-variance-2} holds is exactly the same with $F_{k}^{W}$ instead of $F_{k}$, and it is omitted.
 For $k\in \mathbb{Z} ^{d}$, let $\bar F_{k}=F_{k}(\eta )-\E F_{k}(\eta )$. We have
 $
\notag\E  {(F_{W}-\E F_{W})}^{4}= {\sum_{i,j,k,l\in W}\E \bar {F_{i}}\bar {F_{j}}\bar {F_{k}}\bar {F_{l}}}.
 $
Let $I=\{i,j,k,l\}\subset  W$.  Assume that $i$ is {\it $I$-isolated}, i.e.  $\delta :=[d(i,I\setminus\{i\})]=\max_{m\in I}[d(m,I\setminus \{m\})]$ (let this quantity be $0$ if $i=j=k=l$). Let $  \eta_{m} ,m\in I,$ be  independent copies of $\eta $, and 
$H_{m}=  B_{ \delta /2a_{+}}(m)$,$
\eta' _{m}= (\eta \cap  H_{m} )\cup (  \eta_{m} \cap H_{m}^{c}).$
 Note that $\eta '_{m}$ is distributed as $\eta $,  and that for $m\in I\setminus\{i\},$ $H_{i}\cap H_{m}=\emptyset $, hence $\eta'_{i} $ is independant from    $\{ \eta '_{j},\eta' _{k},\eta' _{l}\}$. Introduce  $\bar F_{m}'=F_{m}(\eta _{m}')-\E F_{m},\bar F=\bar F_{j}\bar F_{k}\bar F_{l},\bar F'=\bar F_{j}'\bar F_{k}'\bar F_{l}'$, independent of $\bar F_{i}'$. We have, using Holder's inequality,
  \begin{align*}
\left|
\E\bar F_{i}  \bar F-\underbrace{\E \bar F'_{i}\bar  F' }_{=0}
\right|\leqslant &\E \Big[
 | (\bar F_{i}-\bar F'_{i} )  \bar F_{j}\bar F_{k}\bar F_{l} | +  | \bar  F_{i}'(\bar F_{j}-\bar F_{j}')\bar F_{k}\bar F_{l} |\\& + |  \bar F'_{i}\bar F_{j}'(\bar F_{k}-\bar F_{k}')\bar F_{l} | + | \bar F'_{i}\bar F'_{j}\bar F'_{k}(\bar F_{l}-\bar F'_{l}) | 
\Big]\\
\leqslant&  4\sum_{m\in I }(\E\bar   F_{0}^{4})^{3/4}(\E | \bar F_{m}-\bar F'_{m} | ^{4}
)^{1/4}\\
\leqslant &\kappa C_{0}(\E \bar F_{0}^{4})^{3/4}(1+\delta )^{-\alpha }
\end{align*}
by \eqref{ass:upper-variance} (or \eqref{ass:upper-variance-2} for the proof with the $F_{k}^{W}$).
 Notice that  one point among $\{j,k,l\}$ is between distance $\delta $ and $\delta +1$ from $i$, call it $a$, and there are at most $\kappa \delta ^{d-1}$ possible values for $a$, given $i$. If there are two points remaining in $\{j,k,l\}\setminus a$, they are at mutual distance at most $3\delta $.  We have 
\begin{align*}
\E ( F_{W}-\E\bar F_{W})^{4}\leqslant &4\sum_{i,j,k,l\in W} \mathbf{1}_{\{i\text{\rm{ isolated}}\}}\kappa C_{0}(\E\bar  F_{0}^{4})^{3/4}(1+[d(i,\{j,k,l\})])^{ -\alpha }\\
\leqslant &\kappa C_{0}(\E\bar  F_{0}^{4})^{3/4}\sum_{\delta =0}^{\infty }(1+\delta )^{ -\alpha }\sum_{i,j,k,l\in W}\mathbf{1}_{\{i\text{\rm{ isolated and }}[d(i,\{j,k,l\})]=\delta \}}\\
\leqslant &\kappa C_{0}(\E\bar  F_{0}^{4})^{3/4}\sum_{\delta =0}^{\infty } | W | ^{2}(1+\delta )^{ -\alpha }\kappa\delta  ^{d-1}(3\delta ) ^{d}\leqslant   \kappa C_{0}(\E\bar  F_{0}^{4})^{3/4} | W | ^{2}
\end{align*}where $\kappa <\infty $ because $\alpha >2d.$

\subsection{Proof of Theorem \ref{thm:BE}}
    \label{sec:BE} 
 $W$ is fixed. For simplicity,  in all the proof we use the notation $G=G_{W},\tilde G=\tilde G_{W}$. If \eqref{eq:moment-derivative} is satisfied, put $G_{k}=F_{k}$ and $A=\mathbb{R}^{d}$. If  instead \eqref{eq:moment-derivative-2}   is satisfied, put $G_{k}=F_{k}^{W} $ and $A=\tilde W$.     
Assume without loss of generality that $F_{0}$  is centered.  
 Theorem 1.2 from \cite{LPS} gives general Berry-Esseen bounds on the Poisson functional $\tilde G $ : provided
 $
\int_{A}\E( D_{\x}G) ^{2}d\x<\infty 
 $
(implied here by Assumption  \eqref{eq:moment-derivative}  or \eqref{eq:moment-derivative-2} and $\alpha >d$),  
 $d_\W  (\tilde G,N)\leqslant \sum_{i=1}^{3}\gamma _{i},\;$$d_\K  (\tilde G,N)\leqslant \sum_{i=1}^{6}\gamma _{i},$ where the $\gamma _{i}$ are quantities depending on the first and second-order Malliavin derivatives of $\tilde G$, whose values are recalled later.
 Let $x,y\in A,\x=(x,M),\y=(y,M')$.
 Call $\eta ^{\x}=\eta \cup \{\x\},\eta ^{\y}=\eta \cup \{\y\}$. We have, using H\"older's inequality at the last line,
\begin{align*}
\left|
D_{\x,\y}G (\eta )
\right|   \leqslant & \sum_{k\in W}\min\left(
 | D_{\x}G _{k}(\eta )| + | D_{\x}G_{k}(\eta^{\y}) |  
,
 | D_{\y}G _{k}(\eta )| + | D_{\y}G_{k} (\eta ^{\x} )| 
\right) . \\
\E  | D^2_{\x,\y}G (\eta ) | ^{4}\leqslant & \E\left|2
\sum_{k\in W}\min\left(
\sup_{\eta'  \in \{\eta ,\eta^{\y}\} }\left|
D_{\x}G_{k}(\eta'   )
\right|,\sup_{ \eta'  \in \{\eta ,\eta ^{\x}\} }\left|
D_{\y}G_{k}(\eta'   )
\right|
\right)
\right|^{4}\\
\leqslant &2^{4}\sum_{k_{1},\dots ,k_{4}\in W} \E \prod_{i=1}^{4}\min\left(
\sup_{\eta' \in \{\eta ,\eta^{\y}\} } | D_{\x}G_{k_{i}}(\eta'   ) | ,\sup_{\eta' \in \{\eta ,\eta ^{\x}\} } | D_{\y}G_{k_{i}}(\eta'   ) | 
\right)\\
\leqslant &2^{4} \left(
\sum_{k\in W}\left(
 \E  \min\left(
\sup_{\eta' \in \{\eta ,\eta^{\y}\}} | D_{\x}G_{k}(\eta'   ) | ^{4},\sup_{\eta'  \in \{\eta ,\eta ^{\x}\} } | D_{\y}G_{k}(\eta'   ) | ^{4}
\right)
\right)^{1/4}
\right)^{4}.
\end{align*}
Let $k\in W$.  Note that, with $B=\tilde W-k$, for $x \in \tilde W,y\in \mathbb{R}^{d}, $
\begin{align*}
D_{\x}F^{W}_{k}(\eta ')=&F^{W}_{k}(\eta '\cup \{\x\})-F^{W}_{k}(\eta ')=F_{0}((\eta '\cup \{\x\})\cap \tilde W-k)-F_{0}(\eta '\cap \tilde W-k)\\
=&F_{0}(((\eta '-k)\cap B)\cup \{\x-k\})-F_{0}((\eta '-k)\cap B)=D_{\x-k}F_{0}((\eta '-k)\cap B).
\end{align*}
Since $\eta  -k\equlaw \eta $, applying either  \eqref{eq:moment-derivative} or \eqref{eq:moment-derivative-2} with $y-k$ instead of $y$ yields
\begin{align*}
\E  | D^{2}_{\x,\y}G (\eta ) | ^{4}\leqslant &\kappa C_{0}^{4}\left(
\sum_{k\in W}\min((1+\|x-k\|)^{-\alpha },(1+\|y-k\|)^{-\alpha })
\right)^{4}.
\end{align*}
Consider case {\bf (i')} (the following is valid but irrelevant in case {\bf (i)}). Summing in a radial manner around $x$ yields that the previous sum is bounded by $\kappa C_{0}^{4}(\sum_{m=[d(x,W)]}^{\infty }m^{d-1}(1+m)^{-\alpha })^{4}\leqslant  \kappa C_{0}^{4}(1+d(x,W))^{4(d-\alpha) }$, and the same holds for $y$. We can also work on the first order derivative with a similar technique:
\begin{align*}
\E  | D_{\x}G | ^{4}\leqslant \kappa C_{0}^{4}\left(
\sum_{k\in W}(1+\|x-k\|)^{-\alpha }
\right)^{4}\leqslant \kappa C_{0}^{4}(1+d(x,W))^{4(d-\alpha )}.
\end{align*}
Noting $I_{x,y}=\{k\in W:\|k-x\|\geqslant \|k-y\|\}$,  
\begin{align*}
\E  | D_{\x,\y}G (\eta ) | ^{4}\leqslant &\kappa C_{0}^{4} \left[
  \left(
\sum_{k\in I_{x,y}}(1+\|x-k\|)^{-\alpha } 
\right)^{4}+ \left(
\sum_{k\in I_{y,x}}(1+\|y-k\|)^{-\alpha } 
\right)^{4}
\right]\\
\leqslant &\kappa C_{0}^{4}\left[
\sum_{k\in \mathbb{Z} ^{d}\setminus B(x,\|y-x\|/2)}(1+\|x-k\|)^{-\alpha }+\sum_{k\in \mathbb{Z} ^{d}\setminus B(y,\|x-y\|/2) }(1+\|y-k\|)^{-\alpha }
\right]^{4}\\
\leqslant &C_{0}^{4}\kappa (1+\|x-y\|/2)^{4(d-\alpha) },
\end{align*}
whence finally
 
\begin{align}
\label{eq:bound-D2}
\E | D_{\x,\y}G (\eta ) | ^{4}\leqslant  & \kappa C_{0}^{4}(1+\max(\|x-y\|,d(x,W),d(y,W)))^{4(d-\alpha) }.
\end{align}
Let us start with a few geometric estimates, useful in the case {\bf (i')}.

\begin{lemma}
Let $W\subset \mathbb{Z} ^{d},$ bounded and non-empty, $d_{W}=(| W | / | \partialZ W | )^{1/d},  W'=\{k\in \mathbb{Z} ^{d}:d(k, W)\leqslant d_{W}\}$. We have
\begin{align}
\label{eq:estimW1} | W' | \leqslant &\kappa  | W | \\
\label{eq:estimW3}\int_{(\tilde W')^{c}}(1+d(x,\tilde W))^{a }dx\leqslant & 
  \kappa _{a} |   W |d_{W}^{a },\;a< -d \\
  \label{eq:estimW2}I(x):=\int_{\mathbb{R}^{d}}(1+\max(d(x ,W),\|x -y\|))^{d-\alpha }dy\leqslant &\kappa (1+d(x ,W))^{2d-\alpha },x\in \mathbb{R}^{d}.
\end{align}

\end{lemma}

\begin{proof}Since each point of $W'\setminus W$ is in a ball with radius $d_{W}$ centered in $\partialZ W$, \eqref{eq:estimW1} is proved via 
\begin{align*}
 | W' | \leqslant  | W | + | \partialZ W | \kappa d_{W}^{d}\leqslant \kappa  | W | .
\end{align*}
Let $\psi (x)=d(x,\tilde W),x\in \mathbb{R}^{d},h(t)=\mathbf{1}_{\{t\geqslant d_{W}\}}(1+t)^{a },t\geqslant 0$. The Federer co-area formula yields
\begin{align*}
\int_{\mathbb{R}^{d}}h(\psi (x))\|\nabla \psi (x)\|dx=\int_{\mathbb{R}_{+}}h(t)\mathcal{H}^{d-1}(\psi ^{-1}(\{t\}))dt.
\end{align*}We have  $\|\nabla \psi (x)\|=1$ for a.a. $x\in \tilde W^{c}$.   According to \cite[Lemma	4.1]{RatWin}, for almost all $t>0,$
\begin{align*}
\mathcal{H}^{d-1}(\psi ^{-1}(\{t\}))\leqslant t^{d-1}\mathcal{H}^{d-1}(\psi ^{-1}(\{1\})),
\end{align*}and the latter is bounded by $\kappa t^{d-1} | \partialZ W | $. Since $a+d<0$

\begin{align*}
\int_{(\tilde W')^{c}}h(\psi (x))dx\leqslant  \kappa \int_{d_{W}}^{\infty }(1+t)^{a }t^{d-1} | \partialZ W | dt\leqslant \kappa_{a} | \partialZ W | d_{W}^{a}d_{W}^{d} =\kappa_{a}  | W | d_{W}^{a},
\end{align*}which yields \eqref{eq:estimW3}. The left hand member of \eqref{eq:estimW2} is equal to 
\begin{align*}
I(x)=&\ell^d(B(x,d(x,W)))(1+\max(d(x,W)))^{d-\alpha }+\int_{B(x,d(x,W))^{c}}(1+\|x-y\|)^{d-\alpha }dy\\
\leqslant& \kappa (1+d(x,W))^{2d-\alpha }+\int_{d(x,W)}^{\infty }(1+r)^{d-\alpha }\kappa r^{d-1}dr,
\end{align*}from which the result follows.
\end{proof}

  Writing $\x_{1}=(x_{1},M_{1}),\x_{2}=(x_{2},M_{2}),\x_{3}=(x_{3},M_{3})$, with $M_{1},M_{2},M_{3}$ iid distributed as $\mu $, denote by $\tilde \E$ the expectation with respect to $(M_{1},M_{2},M_{3})$, and $\E_{\eta }$ the expectation with respect to  $\eta $, such that $\E=\tilde\E\E_{\eta }$.  We have, bounding $\E D_{\x_{1}}^{4}G$ by 
  $\kappa C_{0}^{4}$ and using Cauchy-Scwharz inequality several times,
\begin{align*}
\gamma _{1}=&4\sigma ^{-2}\left[
\int_{ A^{3}}\tilde\E\left[
\sqrt{\E_{\eta }\left[
(D_{\x_{1}}G)^{2}(D_{\x_{2}}G)^{2}
\right]}\sqrt{\E_{\eta }\left[
(D^{2}_{\x_{1},\x_{3}}G)^{2}(D^{2}_{\x_{2},\x_{3}}G)^{2}
\right]}
\right]dx_{1}dx_{2}dx_{3}
\right]^{1/2}\\
\leqslant &4\sigma ^{-2} \left[
\int_{ A^{3}}\sqrt{\tilde\E\left[
\E_{\eta }\left[
(D_{\x_{1}}G)^{2}(D_{\x_{2}}G)^{2}
 \right]
\right]}\sqrt{\tilde\E\left[
\E_{\eta }\left[
(D^{2}_{\x_{1},\x_{3}}G)^{2}(D^{2}_{\x_{2},\x_{3}}G)^{2}
\right]  
\right]}dx_{1}dx_{2}dx_{3}
\right]^{1/2}\\
\leqslant &\kappa C_{0}\sigma ^{-2}\left[
\int_{A^{3}} \left(
\E (D^{2}_{\x_{1},\x_{3}}G)^{4}
\right)^{1/4}\left(
\E (D^{2}_{\x_{2},\x_{3}}G)^{4}
\right)^{1/4}dx_{1}dx_{2}dx_{3}
\right]^{1/2}\\
\leqslant &\kappa C_{0}^{2}\sigma ^{-2}\sqrt{\int_{A}\left
(\int_{A}(1+\max(d(x,\tilde W),\|x-x_{3}\|)))^{d-\alpha }dx
\right)^{2}dx_{3}}
\end{align*}
using \eqref{eq:bound-D2}.  Similar techniques to integrate out the marks yield the same bound
\begin{align*}
\gamma _{2} 
\leqslant &\kappa C_{0}\sigma ^{-2}\left[
\int_{A^{3}} \left(
\E (D^{2}_{\x_{1},\x_{3}}G)^{4}
\right)^{1/4}\left(
\E (D^{2}_{\x_{2},\x_{3}}G)^{4}
\right)^{1/4}dx_{1}dx_{2}dx_{3}\right]^{1/2}\\ 
\leqslant &\kappa C_{0}^{2}\sigma ^{-2}\sqrt{\int_{A}\left
(\int_{A}(1+\max(d(x,\tilde W),\|x-x_{3}\|)))^{d-\alpha }dx
\right)^{2}dx_{3}}
\end{align*}
In the case {\bf (i)}, $A=\tilde W$ and $\alpha >2d$. We have 
\begin{align*}
\max(\gamma _{1},\gamma _{2})\leqslant \kappa C_{0}^{2}\sigma ^{-2}\sqrt{\ell^d(\tilde W)\left(
\int_{\mathbb{R}^{d}}(1+\|  x\|)^{d-\alpha }dx
\right)^{2}}\leqslant \kappa C_{0}^{2}\sigma ^{-2} \sqrt{ | W | }.
\end{align*}In the case ${\bf (i')}, A=\mathbb{R}^{d}, \alpha >5d/2$.
Using successively \eqref{eq:estimW2},\eqref{eq:estimW1} and \eqref{eq:estimW3} yield, with $2(2d-\alpha)<2(-d/2)=-d, $
\begin{align*}
\max(\gamma _{1},\gamma _{2})\leqslant &\kappa C_{0}^{2}\sigma ^{-2}\sqrt{\int_{A}I(x_{3})^{2}dx}
\leqslant\kappa C_{0}^{2}\sigma ^{-2}\sqrt{\kappa  \ell^d(\tilde W')+\int_{(\tilde W')^{c}}(1+d(x,W))^{2(2d-\alpha )}dx}\\
\leqslant &\kappa C_{0}^{2}\sigma ^{-2}\sqrt{\kappa  | W | +\kappa _{2(2d-\alpha )} | W | d_{W}^{2(2d-\alpha )}}\leqslant \kappa C_{0}^{2}\sigma ^{-2} \sqrt{ | W | }(1+d_{W}^{2(2d-\alpha )}),
\end{align*}
which gives the power $a$ in \eqref{eq:BEbound-W-2}-\eqref{eq:BEbound-K}.
Let us keep assuming we are in case {\bf (i')}. Since $A=\mathbb{R}^{d}$ and $\alpha >2d$, \eqref{eq:estimW3} yields
\begin{align*}
\gamma _{3}\leqslant& \sigma ^{-3} \int_{\mathbb{R}^{d}}\left(
C_{0}^{4}\kappa (1+d(x,W))^{4(d-\alpha )}
\right)^{3/4} dx 
\leqslant & C_{0}^{3}\kappa  \sigma ^{-3}  \left(
 \ell^d(\tilde W') +\int_{(\tilde W')^{c}}(1+d(x,W))^{3(d-\alpha )} dx
\right)\\
\leqslant& \kappa  C_{0}^{3}\sigma ^{-3}| W |(1+d_{W}^{3(d-\alpha )}) .\end{align*}
In case {\bf (i)}, the same bound holds after removing $d_{W}^{3(d-\alpha )}$. Reporting back gives \eqref{eq:BEbound-W-2}. 

Introduce $\overline{G }=G -\E G $. Using \eqref{eq:estimW1} and \eqref{eq:estimW3},
\begin{align*}
\gamma _{4}\leqslant &\frac 12\sigma ^{-1}(\E  \overline{G }^{4})^{1/4}\int_{\mathbb{R}^{d}}\sigma ^{-3}\left(
C_{0}^{4}\kappa (1+d(x,W))^{4(d-\alpha )}
\right)^{3/4}dx\\
\leqslant &\kappa \sigma ^{-4}v^{1/4}\sqrt{ | W |} C_{0}^{3} \left(
\ell^d(\tilde W')+\int_{(\tilde W')^{c}}(1+d(x,W))^{3(d-\alpha )} dx
\right)\\
\leqslant & \sigma ^{-4}C_{0}^{3} \kappa | W |^{3/2} v^{1/4}(1+d_{W}^{3(d-\alpha )})
\end{align*}
where
 $
v:=\sup_{ | W | \to \infty }\frac{\E((G -\E G )^{4})}{ | W | ^{2}}
 $. Let us conclude the proof: \eqref{eq:estimW3} yields
\begin{align*}
\gamma _{5}\leqslant &\left[
\int_{\mathbb{R}^{d}}\sigma ^{-4}C_{0}^{4}\kappa (1+d(x,W))^{4(d-\alpha )}dx
\right]^{1/2}\leqslant \sigma ^{-2}C_{0}^{2}\kappa \sqrt{ | W | }\left(
1+d_{W}^{4(d-\alpha )}
\right)^{1/2}\\
\gamma _{6}\leqslant &\Big[
\int_{(\mathbb{R}^{d})^{2}}\sigma ^{-4}\Big(
6C_{0}^{4}\kappa (1+d(x_{1},W))^{2(d-\alpha )}(1+ \|x_{1}-x_{2}\|)^{2(d-\alpha )}\\
&\hspace{4cm}+3C_{0}^{4}\kappa (1+d(x_{1},W))^{2(d-\alpha )}(1+ \|x_{1}-x_{2}\|)^{2(d-\alpha )}
\Big)dx_{1}dx_{2} 
\Big]^{1/2}\\
\leqslant &\sigma ^{-2}C_{0}^{2}\kappa \left[
\int_{\mathbb{R}^{d}}(1+d(x_{1},W))^{2(d-\alpha )}\left(
\int_{\mathbb{R}^{d}}(1+\|x_{1}-x_{2}\|)^{2(d-\alpha )}dx_{2}
\right)dx_{1}\right]^{1/2}\\
\leqslant &\sigma ^{-2}C_{0}^{2}\kappa  \sqrt{ | W | }\left(
1+d_{W}^{2(d-\alpha )}
\right)^{1/2}.
\end{align*}
  In case {\bf (i)}, $A=\tilde W$, all the same inequalities still hold after removing terms of the form $d_{W}^{a}$.
Reporting back gives \eqref{eq:BEbound-K}.

\subsection*{Acknowledgements}

 I am thankful to J. E. Yukich, who brought valuable insights on topics related to asymptotic normality of geometric functionals.
  \bibliographystyle{plain}
 \bibliography{../../../../../../../../bibi2bis}

  \end{document}